\documentclass[12pt]{amsart}
\usepackage{amsmath,amssymb,amsfonts,amsthm,amsopn,dsfont}
\usepackage[dvips]{graphicx}
\usepackage{srcltx}
\usepackage{setspace}
\newtheorem{theorem}{Theorem}[section]

\newtheorem{definition}[theorem]{Definition}

\newtheorem{lemma}[theorem]{Lemma}
\numberwithin{equation}{section}

\newtheorem{proposition}[theorem]{Proposition}
\newtheorem{remark}[theorem]{Remark}

\renewcommand{\ell}{l}
\renewcommand{\epsilon}{\varepsilon}
\def\al{\alpha}                    
\def\be{\beta}
\def\ga{\gamma}

\def\ve{\varepsilon}
\def\eps{\varepsilon}

 \def\cS{\mathcal{S}}

\def\px{\langle x \rangle}

\def\rd{\bR^d}

\def\rdd{{\bR^{2d}}}

\def\R{\right)}

\def\<{\left<}
\def\>{\right>}

\def\mv1{M_v^1}

\def\mn{(m,n)}
\def\mn'{(m',n')}

\newcommand{\eab}{\eps^{|\alpha|+|\beta|}}
\newcommand{\mab}{\max\{|\alpha|,|\beta|\}!}
\newcommand{\tg}{\tilde{\gamma}}
\newcommand{\tid}{\tilde{\delta}}
\newcommand{\sg}{\mathcal{A}_{ sect}(\rd)}

\def\N{\mathbb{N}}
\def\R{\mathbb{R}}
\def\C{\mathbb{C}}
\def\rd{\mathbb{R}^d}
\def\rdd{\mathbb{R}^{2d}}

\begin{document}

\begin{abstract} We prove, for a wide class of
semilinear elliptic
differential and
pseudodifferential equations
in $\R^d$, that the solutions
which are sufficiently
regular and have a certain
decay at infinity extend to
holomorphic functions in
sectors of $\mathbb{C}^d$, improving
earlier results where the
extension was shown for a
strip. Moreover, exponential
decay for such extended
solutions is also proved. The
results apply, in particular,
to solitary wave solutions of
many classical nonlinear evolution
equations.\\

\end{abstract}

\title[Holomorphic extension for elliptic
equations]{Holomorphic extension of
solutions of semilinear elliptic
equations}
\author{Marco Cappiello \and Fabio Nicola}
\address{Dipartimento di Matematica,  Universit\`{a} degli Studi di Torino,
Via Carlo Alberto 10, 10123
Torino, Italy}
\address{Dipartimento di Matematica, Politecnico di
Torino, Corso Duca degli
Abruzzi 24, 10129 Torino,
Italy}
\email{marco.cappiello@unito.it}
\email{fabio.nicola@polito.it}
\date{}

\maketitle

\section{Introduction}
The main concern in this paper is the study of holomorphic extensions of the solutions of semilinear elliptic equations in $\R^{d}$. Broadly speaking, we deal with
equations of the form \begin{equation} \label{generalequation}
Pu=F[u],
\end{equation}
where $P$ is a linear elliptic differential, or even pseudodifferential, operator in $\R^{d}$ and $F[u]$ is a nonlinearity, possibly involving the derivatives of $u$. For a wide class of equations of this type it is known that every solution $u$ sufficiently regular and with a certain decay at infinity, actually is analytic on $\R^{d}$ and it extends to a holomorphic function in a strip of $\C^{d}$ of the form
\begin{equation}\label{in0}
\{z=x+iy\in\mathbb{C}^d:\
|y|\leq\eps\},
\end{equation}
for some $\eps>0$, satisfying there the
estimates
\begin{equation}\label{in0bis}
|u(x+iy)|\leq C e^{-c|x|}
\end{equation}
for some $C>0$, $c>0$. A
pioneering work on this
subject was the paper by Kato
and Masuda \cite{kato}. Later
the problem of the
holomorphic extension in a
strip has been intensively
studied in connection with
the applications to solitary
wave equations. In particular
it was noticed in dimension 1
that several model equations
like the Korteweg-de Vries
equations and its
generalizations,
Schr\"odinger-type equations,
long wave-type equations
admit solitary wave solutions
which extend to meromorphic
functions with poles out of a
strip of the form \eqref{in0}
and having a decay of type
\eqref{in0bis}. Among the
main contributions in this
sense we recall the papers by
Bona and Li \cite{BL1},
\cite{BL2}, Grujic' and
Kalisch \cite{gru}, Bona,
Grujic' and Kalisch
\cite{bo1}, Bona and Weissler
\cite{bo3}. We also recall
the paper \cite{hay} by
Hayashi on well-posedness of
the generalized KdV equation
in Bergman-Szeg\"o spaces of
holomorphic functions in a
sector. Apart from its
interest ``per se'' in the
analysis of regularity
properties of the solutions
of differential equations,
the study of {\it complex}
singularities of solitary
waves could give information
also on the onset of real
blowup; we refer to Bona and Weissler
\cite{bo3} for a fascinating
discussion in this
connection.\par Recently, the
properties described above
have been proved for some
general classes of semilinear
elliptic equations in any
dimension, even with variable
coefficients; see for example
Biagioni and Gramchev \cite{A3}, Gramchev \cite{gramchev} and Gramchev, Cappiello and Rodino \cite{A18,A38,A39, CGR, CGR2}. The
results in these papers have
been stated and proved in
terms of estimates in the
Gelfand-Shilov spaces of type S, cf.
\cite{GS}.  We refer to Nicola and Rodino \cite[Chapter 6]{nicola} for a self-contained account of these results.\par
Nevertheless
in the above mentioned papers
some relevant issues remained
unexplored. The first one is
the identification of a
\textit{maximal} holomorphic
extension. In other words,
the problem is to understand
what is the biggest complex domain on which a
holomorphic extension is
possible. The second one is
related to a dual aspect and
concerns the identification
of the maximal domain on which the decay
properties on $\R^{d}$ of
solutions remain valid.
\par To be precise, let us introduce
a class of operators which
the above results and those
in the present paper apply
to. First, one can consider
differential operators with
polynomial coefficients
\begin{equation}\label{in1bis}
P=\sum_{|\alpha|\leq m,\,|\beta|\leq n}
c_{\alpha\beta}x^\beta D^\alpha,
\end{equation}
$m\geq 1$, $n\geq0$,
$c_{\alpha\beta}\in\mathbb{C}$,
with symbol
\begin{equation}\label{in1bisbis}
p(x,\xi)=\sum_{|\alpha|\leq
m,\,|\beta|\leq n}
c_{\alpha\beta}x^\beta
\xi^\alpha
\end{equation}
of {\it $G$-elliptic} type,
namely satisfying the
following global version of
ellipticity:
\begin{equation}\label{in1}
\sum_{|\alpha|=m,\,|\beta|\leq n}
c_{\alpha\beta}x^\beta
\xi^\alpha\not=0,\quad x\in\rd,\
\xi\not=0,
\end{equation}
\begin{equation}\label{in2}
\sum_{|\alpha|\leq m,\,|\beta|= n}
c_{\alpha\beta}x^\beta
\xi^\alpha\not=0,\quad x\not=0,\
\xi\in\rd,
\end{equation}
\begin{equation}\label{in3}
\sum_{|\alpha|=m,\,|\beta|= n}
c_{\alpha\beta}x^\beta
\xi^\alpha\not=0,\quad x\not=0,\
\xi\not=0.
\end{equation}
If $P$ has constant
coefficients (that is $n=0$)
then \eqref{in1},
\eqref{in2}, \eqref{in3} are
satisfied if and only if $P$ is
elliptic and its symbol
$p(\xi)$ satisfies
$p(\xi)\not=0$ in $\rd$. As a
relevant model one can
consider the operator
$P=-\Delta + \lambda$, $
\lambda\in\mathbb{C},\,\lambda\not\in\mathbb{R}_{-}\cup\{0\}$.\par
More generally we deal with pseudodifferential
operators. Namely, given real
numbers $m
>0$, $n\geq 0$ we consider symbols
$p(x,\xi)$ satisfying the
following estimates
\begin{equation}
\label{Gsymbols}
|\partial_{\xi}^{\alpha}\partial_{x}^{\beta}p(x,\xi)|\leq
C^{|\alpha|+|\beta|+1}\alpha!
\beta!\langle \xi
\rangle^{m-|\alpha|}\langle x
\rangle^{n-|\beta|}
\end{equation} for every
$(x,\xi) \in \R^{2d}, \alpha
, \beta \in \N^{d}$ and for
some positive constant $C$
independent of $\alpha,
\beta$  (we denote as usual $\langle
x\rangle=(1+|x|^2)^{1/2}$). Notice that every
polynomial $p(x,\xi)$ as in
\eqref{in1bisbis} satisfies
such estimates. For this
general class the condition
of $G$-ellipticity can be
stated by requiring that
\begin{equation}
\label{Gell} |p(x,\xi)|\geq c
\langle \xi \rangle^{m}
\langle x \rangle^{n}, \qquad
|x|+|\xi|\geq C,
\end{equation}
for some positive constants
$c,\,C$. It is known that a
symbol of the form
\eqref{in1bisbis} satisfies
\eqref{Gell} if and only if
it satisfies the three
conditions \eqref{in1},
\eqref{in2}, \eqref{in3}
simultaneously. \par
Finally, concerning the
nonlinear term $F[u]$ in
\eqref{generalequation}, we
assume here that it is of the
form
\begin{equation}
\label{A5.2}
F[u]=\sum_{h,l}\sum_{\rho_1,\ldots,\rho_l}
F_{h,l,\rho_1,\ldots,\rho_l}
x^h \prod_{k=1}^l
\partial^{\rho_k}u,
\end{equation}
where $h\in\N^d$, with
$0\leq|h|\leq\max\{n-1,0\}$,
$\rho_k\in\N^d$ with
$0\leq|\rho_k|\leq
\max\{m-1,0\}$, $l\in\N$,
$l\geq2$ and
$F_{h,l,\rho_1,\ldots,\rho_l}\in\mathbb{C}$
(the above sum being finite).
Moreover, we allow some of
the factors in \eqref{A5.2}
to be replaced by their
complex conjugates. \par As a 
simple example, consider the
equation
\begin{equation}\label{KG}
-\Delta u+u=|u|^{l-1}u,
\end{equation}
with $l\in\mathbb{N}$, $l>2$ odd, which
arises e.g. when looking for standing
wave solutions for the Klein-Gordon or
Schr\"odinger equation, as well as
travelling wave solutions for the
Klein-Gordon equation (cf. Berestycki
and Lions \cite{A35}). The existence of
solutions in $H^1(\rd)$ and their
exponential decay was studied in
\cite{A35}, whereas the possibility of
extending them holomorphically on a
strip has been recently shown in
\cite{A39} (incidentally, the
exponential decay generally drops for
elliptic equations which are not
globally elliptic, such as $-\Delta
u=|u|^{l-1}u$).  \par
 The above
more general class of operators
includes, in dimension $d=1$, the
solitary wave counterpart of several
evolution equations of Korteweg-de
Vries type, as well as of long-wave
type; see \cite{BL2} and below. This
also motivated the statements of the
results for pseudodifferential
operators, or at least for Fourier
multipliers (e.g. $p(\xi)=\xi\,{\rm
Coth}\,\xi+\lambda$, $\lambda>-1$, for
the intermediate-long-wave equation;
see \cite{BL2}).\par Now, it is known
from \cite{BL2} and \cite[Theorem
7.3]{A39} that (for classes of
nonlinearities that intersect the above
one), if $P$ is $G$-elliptic, then all
the solutions
 $u$ of \eqref{generalequation} with $u \in H^{s}(\R^{d})$,
 $s>\frac{d}{2}+m-1$ for $n>0$
 (respectively $\langle x \rangle^{\ve_{o}}u \in H^{s}(\R^{d})$,
 $s>\frac{d}{2}+m-1$ for $n=0$)
  actually decay at infinity like $e^{-c|x|},c>0$,
  and extend holomorphically on a strip of the form
  \eqref{in0}.\par
Here we shall improve this
result by showing that the
holomorphic extension and
the exponential decay actually
hold in a {\it sector} of the
complex domain. Namely, we
have the following result,
which seems new even
for the simplest equation
\eqref{KG}.

\begin{theorem}
\label{mainthm} Let $P$ be a
pseudodifferential operator
with symbol $p$ satisfying
\eqref{Gsymbols}, $m>0$,
$n\geq 0$, and assume that $p$
is $G$-elliptic, that is
\eqref{Gell} is satisfied.
Let $F[u]$ be of the form
\eqref{A5.2} (possibly with
some factors in the product
replaced by their
conjugates). Assume moreover
that $u \in H^s(\rd)$,
$s>\frac{d}{2}+\max\{|\rho_k|\}$,
is a solution of
\eqref{generalequation}. In
the case $n=0$ assume further
$\langle
x\rangle^{\epsilon_0} u\in
L^2(\rd)$, for some
$\epsilon_0>0$. Then $u$
extends to a holomorphic
function in the sector of
$\mathbb{C}^d$
\begin{equation}\label{in6}
\mathcal{C}_\eps=\{z=x+iy\in\mathbb{C}^d:\
|y|\leq\eps(1+|x|)\},
\end{equation}
for some $\eps>0$, satisfying there the
estimates \eqref{in0bis} for some
constants $C>0, c>0$.
\end{theorem}

The width of the
sector depends in general on
the solution $u$ considered;
in fact, for autonomous
equations the width {\it
must} depend on $u$, because
one can exploit the
invariance with respect to translations to
move the complex
singularities parallel to
$\rd$.\par The shape of the
domain  of holomorphic
extension as a sector is
 sharp, in the sense that, even in dimension $d=1$, for any
angle $\theta\not=0,\pi$
we can find $G$-elliptic equations with constant
coefficients admitting
Schwartz solutions whose
meromorphic extensions have a
sequence of poles along the
ray ${\rm arg}\,z=\theta$. We
refer to Section
\ref{secesempi} below for
details on this point and also for remarks on the a priori regularity assumptions in Theorem \ref{mainthm}.
\par The linear case
($F[u]=0$) deserves a separate
discussion. Indeed, the
analysis of the linear
equation $Pu=0$ is important
for the holomorphic extension
and the decay of
eigenfunctions of
$G$-elliptic operators and
their powers; see Maniccia and Panarese \cite{MP} and Schrohe
\cite{Schrohe:1}. Moreover in
the linear case it is
possible to relax the
assumptions on the regularity
and decay of $u$, admitting
solutions with a priori
algebraic growth and in that
case the width of the sector
is independent on the
particular solution
considered; see Theorem
\ref{linearthm} below.\par
Finally we present an
application of the above
result to solitary waves.
Following \cite{BL2}, we
consider the following class
of Korteweg-de Vries-type
equation
\begin{equation}\label{kdvt}
v_t+v_x+F[v]_x-(Mv)_x=0,
\quad (t,x)\in\R\times\R,
\end{equation}
and long-wave-type equations
\begin{equation}\label{lwt}
v_t+v_x+F[v]_x+(Mv)_t=0,
\quad (t,x)\in\R\times\R,
\end{equation}
where $M=p(D)$ is a Fourier
multiplier, $F[v]$ is a
polynomial with real coefficients,
$F(0)=F'(0)=0$,  and
subscripts denote
derivatives. We look for
solutions $v(t,x)$ of
solitary wave type, i.e.
$v(t,x)=u(x-Vt)$, for some
function $u$ of one variable
and some constant velocity
$V$. We have the following
result.
\begin{theorem}\label{applthm}
Let $p(\xi)$ satisfy the analytic
symbol estimates of order
$m\in\R$, namely
\begin{equation}\label{simbest}
|\partial^{\alpha}
p(\xi)|\leq
A^{\alpha+1}\alpha!\langle
\xi\rangle^{m-\alpha},\quad
\xi\in\R,\
\alpha\in\mathbb{N}
\end{equation}
for some constant $A>0$, as
well as the lower bounds
\begin{equation}\label{plm0}
p(\xi)\geq0,\ \xi\in\R,\qquad
p(\xi)\geq C^{-1}|\xi|^m,\ |\xi|\geq C,
\end{equation}
for some constant $C>0$.
Suppose moreover
$m\geq1$.\par
 Let
$v(t,x)=u(x-Vt)$ be a weak
solution of \eqref{kdvt} or
\eqref{lwt}, with $V>1$,
$u\in L^\infty(\R)$,
$\lim_{x\to\infty} u(x)=0$.
Then $u$ extends to a
holomorphic function $u(x+iy)$ in the
sector
\begin{equation}\label{sector}
\{z=x+iy\in\mathbb{C}:\
|y|\leq\eps(1+|x|)\}
\end{equation} for some $\eps>0$,
satisfying there the
estimates \eqref{in0bis} for
some constants $C>0, c>0$.
\end{theorem}
Notice that the estimates
\eqref{simbest} are satisfied
by any polynomial $p(\xi)$ of
degree $m$. More generally,
the condition \eqref{simbest} is
equivalent to saying that
$p(\xi)$ extends to a
holomorphic function
$p(\xi+i\eta)$ in a sector of
the type \eqref{sector}, and
satisfies there the bounds
$|p(\xi+i\eta)|\leq
C'\langle\xi\rangle^{m}$ (see
Proposition \ref{carat}
below). This remark makes
\eqref{simbest} very easy to
check in concrete situations,
where typically $p(\xi)$ is
expressed in terms of
elementary functions.\par
 We also
observe that, under the hypotheses of
Theorem \ref{applthm}, we already know
from \cite{BL2} that $u$ extends to a
holomorphic function on a strip and
displays there an exponential decay of
type \eqref{in0bis}, hence Theorem
\ref{applthm} can be regarded as an
improvement of this result. For the
{\it existence} of solitary waves for
these equations we refer to the
detailed analysis in Albert, Bona and
Saut \cite{albert}, Amick and Toland
\cite{amick}, Benjamin, Bona and Bose
\cite{benjamin}, Weinstein
\cite{weinstein}. \par Finally we
mention that similar extensions results
should hopefully be valid for other
classes of non-linear elliptic
equations, e.g. with linear part
$P=-\Delta+|x|^2$ (cf. \cite{A18}).
Similarly, even non-elliptic
hypoelliptic operators and dispersion managed solitons could share
similar properties, possibly with the
sector replaced by a smaller domain
(but larger than a strip). We plan to
investigate these issues in a
subsequent paper.\par\medskip The paper
is organized as follows. Section
\ref{prelimi} collects notation and
some preliminary results about
$G$-pseudodifferential operators
(composition, boundedness on Sobolev
spaces, parametrices, etc.). In Section
\ref{secspazi} we introduce a suitable
space of analytic functions on $\rd$,
which admit a holomorphic extension to
sectors in $\mathbb{C}^d$, and we prove
some relevant properties used in the
sequel. In Section \ref{secdim} we
prove Theorem \ref{mainthm}. The proof
is based on the application of an
iterative Picard scheme in the space of
analytic functions defined in Section
\ref{secspazi}. In Section
\ref{secesempi} we prove Theorem
\ref{applthm} and treat in detail the
linear case $F[u]=0$, see Theorem
\ref{linearthm}. We moreover show some
other examples and counterexamples and
test on them the sharpness of our
results.

\section{Notation and preliminary
results}\label{prelimi}
\subsection{Factorial and
binomial coefficients} We use
the usual multi-index
notation for factorial and
binomial coefficients. Hence,
for
$\alpha=(\alpha_1,\ldots,\alpha_d)\in\mathbb{N}^d$
we set
$\alpha!=\alpha_1!\ldots\alpha_d!$
and for
$\beta,\alpha\in\mathbb{N}^d$,
$\beta\leq\alpha$, we set
$\binom{\alpha}{\beta}=\frac{\alpha!}{\beta!(\alpha-\beta)!}$.\par
The following inequality is
standard and used often in
the sequel:
\begin{equation}\label{pr6}
\binom{\alpha}{\beta}\leq
2^{|\alpha|}.
\end{equation}
Also, we recall the identity
$$ \sum_{|\alpha'|=j\atop \alpha'\leq\alpha}\binom{\alpha}{\alpha'} = \binom{|\alpha|}{j}, \qquad j=0,1,\ldots, |\alpha|,$$
which follows from $\prod _{i=1}^d
(1+t)^{\alpha_i}=(1+t)^{|\alpha|}$, and
gives in particular
\begin{equation}\label{pr4}
\binom{\alpha}{\beta}\leq\binom{|\alpha|}{|\beta|},\quad
\alpha,\beta\in\mathbb{N}^d,\
\beta\leq\alpha.
\end{equation}
The last estimate implies in turn, by induction,
\begin{equation}\label{pr8}
\frac{\alpha!}{\delta_1!\ldots\delta_j!}\leq
\frac{|\alpha|!}{|\delta_1|!\ldots|\delta_j|!},\quad\alpha=\delta_1+\ldots+\delta_j,
\end{equation}
as well as
\begin{equation}\label{pr5}
\frac{\alpha!}{(\alpha-\beta)!}\leq
\frac{|\alpha|!}{|\alpha-\beta|!},\quad
\beta\leq\alpha.
\end{equation}
Finally we recall the
so-called inverse Leibniz'
formula:
\begin{equation}\label{pr3}
x^\beta\partial^\alpha
u(x)=\sum_{\gamma\leq\beta,\,\gamma\leq\alpha}\frac{(-1)^{|\gamma|}\beta!}{(\beta-\gamma)!}
\binom{\alpha}{\gamma}\partial^{\alpha-\gamma}(x^{\beta-\gamma}u(x)).
\end{equation}
\subsection{$G$-Pseudo-differential
operators}
Pseudo-differential operators
are formally represented as
integral operators of the
type
\begin{equation}\label{pr0}
p(x,D)u(x)=(2\pi)^{-d}\int_{\R^d}
e^{ix\xi}
p(x,\xi)\widehat{u}(\xi)\,d\xi,
\end{equation}
where \[
\widehat{u}(\xi)=\int_{\R^d}
e^{-ix\xi} u(x)\,dx
\]
 denotes
the Fourier transform of $u$
and $p(x,\xi)$ is the
so-called {\it symbol} of
$p(x,D)$. According to the
symbol spaces which $p$
belongs to, one can consider
$u$ in several classes of
functions or distributions
and symbolic calculi and
boundedness results on
Sobolev spaces are available.
We briefly recall this for
the class of the so-called
$G$-pseudodifferential
operators (also na\-med $SG$
or {\it scattering}
pseudodifferential operators
in the literature). They are
defined by the formula
\eqref{pr0}, where $p(x,\xi)$
satisfies, for some
$m,n\in\R$,
 the following
estimates: for every
$\alpha,\beta\in\mathbb{N}^d$
there exists a constant
$C_{\alpha,\beta}>0$ such
that
\begin{equation}\label{defing}
|\partial^\alpha_\xi\partial^\beta_x
p(x,\xi)|\leq
C_{\alpha,\beta}\langle
x\rangle^{n-|\beta|}\langle\xi\rangle^{m-|\alpha|}
\end{equation}
for every $x,\xi\in\rd$.
The space of functions
satisfying these estimates is
denoted by
 $G^{m,n}(\rd)$, whereas we set
 ${\rm OP}G^{m,n}(\rd)$ for
 the corresponding
 operators.{}{}
 We endow $G^{m,n}(\rd)$ with the topology defined by the seminorms
 \[
 \|p\|^{(G)}_N=\sup_{|\alpha|+|\beta|\leq
 N}\sup_{(x,\xi)\in\rdd}\big\{|\partial^\alpha_\xi\partial^\beta_x
p(x,\xi)|\langle
x\rangle^{-n+|\beta|}\langle\xi\rangle^{-m+|\alpha|}\big\},\quad
N\in\mathbb{N}.
\]
As a prototype one can take
$P=-\Delta+\lambda$,
$\lambda\in\mathbb{C}$, of order
$m=2$, $n=0$. More generally,
the case of Fourier
multipliers, where the symbol
$p(\xi)$ depends only on
$\xi$ (hence $n=0$) is of
great interest, mostly for
applications to solitary
waves and ground state
equations, see \cite{A39}. As
another example, we have
$x^\beta\partial^\alpha
\in{\rm
OP}G^{|\alpha|,|\beta|}(\rd)$.
\par The classes ${\rm OP}G^{m,n}(\rd)$
 were introduced
 in \cite{rD1} and studied in detail in \cite{rD2}, \cite{ES}, \cite{Me}, \cite{Schrohe:2}.
They are in fact a particular case of
the general
  H\"ormander's classes, see
  \cite[Chapter XVIII]{hormanderIII}, and turn out to be very convenient for
 a series of problems involving global aspects of partial
 differential equations in
  $\rd$.\par
   We now summarize
  some properties which
will be useful for us later
on; beside the above mentioned papers, we
refer to \cite[Chapter 3]{nicola} for a detailed and self-contained presentation.\par
First, if $p\in G^{m,n}(\rd)$
then $p(x,D)$ defines a
continuous map
$\cS(\rd)\to\cS(\rd)$ which
extends to a continuous map
$\cS'(\rd)\to\cS'(\rd)$. The
composition of two such
operators is therefore well
defined in $\cS(\rd)$ and in
$\cS'(\rd)$; more precisely,
if $p_1\in G^{m_1,n_1}(\rd)$
and $p_2\in G^{m_2,n_2}(\rd),$
then
$p_1(x,D)p_2(x,D)=p_3(x,D)$
with $p_3\in
G^{m_1+m_2,n_1+n_2}(\rd)$ and
the map $(p_1,p_2)\mapsto
p_3$ is continuous
$G^{m_1,n_1}(\rd)\times
G^{m_2,n_2}(\rd)\to
G^{m_1+m_2,n_1+n_2}(\rd)$.
\par
A symbol $p\in G^{m,n}(\rd)$
(and the corresponding
operator) is called {\it
$G$-elliptic} if for some
constants $C,c>0$, it
satisfies
\begin{equation}\label{Gell2}
|p(x,\xi)|\geq c\langle
x\rangle^n\langle\xi\rangle^m,\quad
|x|+|\xi|\geq C.
\end{equation}
For example,
$P=-\Delta+\lambda$ is
$G$-elliptic if and only if
$\lambda\not\in\R_{-}\cup\{0\}$.
More generally, for an
operator with polynomial
coefficients as in
\eqref{in1bis},
$G$-ellipticity is equivalent
to \eqref{in1}, \eqref{in2},
\eqref{in3} to hold
simultaneously.\par
$G$-ellipticity guaranties
the existence of a parametrix
$E\in{\rm OP}G^{-m,-n}(\rd)$
of $P=p(x,D)$, namely
$EP=I+R$ and $PE=I+R'$, where
$R,R'$ are (globally)
regularizing
pseudodifferential
operators, i.e. with Schwartz
symbols. Hence $R$ and $R'$
are continuous maps
$\cS'(\rd)\to\cS(\rd)$.\par
It is useful to consider the
action of such operators
on the standard
Sobolev spaces
\[
H^{s}(\rd)=\{ u\in\cS'(\rd):\
\|u\|_s:=\left(\int
|\widehat{u}(\xi)|^2
(1+|\xi|^2)^s\,d\xi\right)^{1/2}<\infty\},
\]
and on the weighted versions
\[
H^{s_1,s_2}(\rd)=\{
u\in\cS'(\rd):\
\|u\|_{s_1,s_2}:=\|\langle
x\rangle^{s_2}
u\|_{s_1}<\infty\}.
\]
Indeed, if $p\in
G^{m,n}(\rd)$ then
\begin{equation}\label{contin}
p(x,D):H^{s_1,s_2}(\rd)\to
H^{s_1-m,s_2-n}(\rd)
\end{equation}
continuously, and
\[
\|p(x,D)\|_{\mathcal{B}(H^{s_1,s_2}(\rd),H^{s_1-m_1,s_2-m_2}(\rd))}\leq
C\|p\|^{(G)}_{N}
\]
for suitable $C>0$,
$N\in\mathbb{N}$ depending
only on $s_1,s_2,m_1,m_2$ and on the dimension $d$. In
particular, for $s_2=0$ we
see that, if $n\leq0$ then
$p(x,D):H^{s}(\rd)\to
H^{s-m}(\rd)$ continuously
for every $s\in\R$.

\begin{remark} \label{rem1} The
complex interpolation for the spaces
$H^{s_1,s_2}(\rd)$ works as
one expects, i.e. for
$s_1,s_2,t_1,t_2\in\R$,
$0<\theta<1$,
\begin{equation}\label{interp}
[H^{s_1,t_1}(\rd),H^{s_2,t_2}(\rd)]_\theta=H^{s,t}(\rd),\quad
s=(1-\theta)s_1+\theta s_2,\
t=(1-\theta)t_1+\theta t_2,
\end{equation}
see for example \cite{F-G,triebel}. The property \eqref{interp} will be useful in the sequel in
view of the following consequence: {\it
suppose that $u\in H^s(\rd)$ and
$\langle x\rangle^{\eps_0} u\in
L^2(\rd)$ for some $\eps_0>0$. Then for
every $s'<s$ there exists $\eps>0$ such
that $\langle x\rangle^\eps u\in
H^{s'}(\rd)$.}
\end{remark}
  Finally
we point out for further
reference the following
formulas, which can be
verified by a direct
computation: for
$\alpha,\beta\in\mathbb{N}^d$,
\begin{equation}\label{pr1}
x^\beta
Pu=\sum_{\gamma\leq\beta}
(-1)^{|\gamma|}\binom{\beta}{\gamma}(D^\gamma_\xi
p)(x,D)(x^{\beta-\gamma}u),
\end{equation}
\begin{equation}\label{pr2}
\partial^\alpha Pu=\sum_{\delta\leq\alpha}
\binom{\alpha}{\delta}(\partial^\delta_x
p)(x,D)\partial^{\alpha-\delta}u.
\end{equation}
\section{A space of analytic
functions}\label{secspazi} We
introduce here a space of
analytic functions in $\R^d$, already considered
 in \cite{A38} (and denoted there by $S^{1\star}_1(\rd)$),
which is tailored to the
problem of the holomorphic
extension to subsets of
$\mathbb{C}^d$ of the form
\eqref{in6}, as shown by the
subsequent
 Theorem \ref{estensione}.
\begin{definition}\label{classisg}
We denote by $\sg$ the space of all
functions $f\in C^\infty(\rd)$
satisfying the following condition:
there exists a constant $C>0$ such that
\begin{equation}\label{cla0}
|x^\beta\partial^\alpha f(x)|\leq
C^{|\alpha|+|\beta|+1}\mab, \quad{\rm for\ all}\ \alpha,\,\beta\in\N^d.
\end{equation}
\end{definition}
\begin{theorem}\label{estensione}
Let $f\in\sg$. Then $f$ extends to a
holomorphic function $f(x+iy)$ in the
sector of $\mathbb{C}^d$
\begin{equation}\label{cla3}
\mathcal{C}_\eps=\{z=x+iy\in\mathbb{C}^d:\ |y|\leq\eps(1+|x|)\}
\end{equation}
for some $\eps>0$, satisfying there the estimates
\begin{equation}\label{cla4}
|f(x+iy)|\leq Ce^{-c|x|},
\end{equation}
for some constants $C>0$, $c>0$.
\begin{proof}
First we show the estimates
\begin{equation}\label{cla5}
|x^\beta\partial^\alpha f(x)|\leq C^{|\alpha|+1}|\alpha|! e^{-c|x|},\quad {\rm for}\ |\beta|\leq|\alpha|.
\end{equation}
Indeed, since $|x|^n\leq
k^n\sum_{|\gamma|=n}|x^\gamma|$
for a constant $k>0$
depending only on the
dimension $d$, by
\eqref{cla0} we have
(assuming $C\geq1$ in
\eqref{cla0})
\begin{align*}
e^{c|x|} |x^\beta\partial^\alpha
f(x)|&=\sum_{n=0}^\infty
\frac{(c|x|)^n}{n!}|x^\beta\partial^\alpha
 f(x)|\\
 &\leq \sum_{n=0}^\infty (ck)^n\sum_{|\gamma|=n}\frac{1}{|\gamma|!}|x^{\beta+\gamma}\partial^\alpha f(x)|\\
 &\leq \sum_{n=0}^\infty (ck)^n\sum_{|\gamma|=n} C^{2|\alpha|+|\gamma|+1}\frac{(|\alpha|+|\gamma|)!}{|\gamma|!}.
 \end{align*}
 Since the number of multi-indices $\gamma$ satisfying
  $|\gamma|=n$ does not exceed $2^{d+n-1}$, an application of \eqref{pr6} gives \eqref{cla5} for a new constant $C$, if $c$ is small enough.\par
 Now, \eqref{cla5} and the estimate $|\alpha|!\leq d^{|\alpha|}\alpha!$ give
  \begin{equation}\label{cla6}
 |\partial^\alpha f(x)|\leq C^{|\alpha|+1}\alpha!\langle x\rangle^{-|\alpha|}e^{-c|x|},
 \end{equation}
  for a new constant $C>0$. This shows that the power series
  \begin{equation}\label{cla7}
  \sum_{\alpha}\frac{\partial^\alpha f(x)}{\alpha!}(z-x)^\alpha,
  \end{equation}
  for any fixed $x\in\R^d$ converges in a polydisc in $\mathbb{C}^d$ defined by $|z_k-x_k|<\frac{\langle x\rangle}{2C}$, $1\leq k\leq d$. The union of such polydiscs, when $x$ varies in $\R^d$, cover a subset $\mathcal{C}_\eps\subset\mathbb{C}^d$ of the type \eqref{cla3}, for some $\eps>0$. Since on the
   intersection (when not empty) of two such polydiscs
   these extensions agree ($\rd\subset\mathbb{C}^d$ is totally real), $f(x)$ extends to a holomorphic function
    on $\mathcal{C}_\eps$. For $z\in \mathcal{C}_\eps$,
    using the representation \eqref{cla7} as a power
     series with $x={\rm Re}\,z$ and \eqref{cla6},
      we also get the desired estimate \eqref{cla4} for a new constant $C>0$.
\end{proof}
\end{theorem}
In the sequel we will use the
following characterization of
the space $\sg$ in terms of
$H^s$-based norms.\par Set,
for $f\in\cS'(\rd)$,
\begin{equation}\label{accaenne2}
S_\infty^{s,\eps}[f]=\sum_{\al,\,\beta}\frac{\eab}{\mab} \|x^\be\partial^\al f\|_s.
\end{equation}
\begin{proposition}\label{stima0}
Let $f\in\sg$. Then for every
$s\geq0$ there exists $\eps>0$ such
that $S^{s,\eps}_\infty[f]<\infty$.\par
In the opposite direction, if for some $s\geq0$ there exists $\eps>0$ such that $S^{s,\eps}_{\infty}[f]<\infty$ then $f\in\sg$.
\end{proposition}
\begin{proof}
Assume $f\in\sg$. It suffices to argue when $s=k$ is integer.
Then
\[
\|x^\be\partial^\al f\|_k\leq
C'\sum_{|\gamma|\leq
k}\|\partial^\gamma\big(x^\be\partial^\al
f\big)\|_{L^2}.
\]
Now, if $M\in\mathbb{N}$ satisfies
$M>d/4$ we have
\begin{equation}\label{cla2}
\|\partial^\gamma\big(x^\be\partial^\al
f\big)\|_{L^2}\leq
C''\|(1+|x|^2)^M\partial^\gamma\big(x^\be\partial^\al
f\big)\|_{L^\infty}.
\end{equation}
On the other hand we have, by Leibniz' formula and \eqref{cla0},
\begin{align*}
\|(1+|x|^2)^M\partial^\gamma\big(x^\be\partial^\al
f\big)\|_{L^\infty}&
\leq\sum_{\sigma\leq\gamma,\,\sigma\leq\beta}
\binom{\gamma}{\sigma}\frac{\beta!}{(\beta-\sigma)!}
\|(1+|x|^2)^M
x^{\be-\sigma}\partial^{\al+\gamma-\sigma}
f\|_{L^\infty}\\
&\leq C_\gamma|\beta|^{|\gamma|} C^{|\alpha|+|\beta|}\max\{2M+|\beta|,|\alpha|+|\gamma|\}!.
\end{align*}
Since
$\max\{2M+|\beta|,|\alpha|+|\gamma|\}\leq
\max\{|\beta|,|\alpha|\}+2M+|\gamma|$ and
$|\beta|^{|\gamma|}\leq
\tilde{C}_\gamma^{|\beta|}$, by
\eqref{pr6}
we get
\[
\|x^\be\partial^\al f\|_k\leq C_k^{|\alpha|+|\beta|+1}\mab
\]
for some constant $C_k>0$. Hence $S^{s,\eps}_\infty[f]<\infty$ if $\eps<C_k^{-1}$.\par
In the opposite direction, we may take $s=0$; hence assume $S^{0,\eps}_\infty[f]<\infty$ for some $\eps>0$. Then \eqref{cla0} holds with the $L^\infty$ norm replaced by the $L^2$ norm. If $M$ is an integer, $M>d/2$, we have
\[
\|x^\beta\partial^\alpha f\|_{L^\infty}\leq C\sum_{|\gamma|\leq M} \|\partial^\gamma\big(x^\beta\partial^\alpha f\big)\|_{L^2}.
\]
Hence an application of Leibniz' formula and the same arguments as above show that $f\in\sg$.
\end{proof}


\section{Proof of the main result (Theorem
\ref{mainthm})}\label{secdim}
In this section we prove
Theorem \ref{mainthm}. In
fact we shall state and prove
 this result for the more general non-homogeneous equation
\begin{equation}
\label{A5.1f} Pu=f+F[u],
\end{equation}
where $P$  and $F[u]$ satisfy
the assumptions of Theorem
\ref{mainthm} and $f$ is a
function in the space $\sg$ defined in Section 3.
Moreover we shall restate our
result in terms of estimates in $\sg$. Namely,
 in view of Theorem
 \ref{estensione},
it will be sufficient to
prove the following theorem.

\begin{theorem}\label{AA5.1}
Let $P$ be a
pseudodifferential operator
with symbol $p$ satisfying
\eqref{Gsymbols}, $m>0, n\geq
0$ and assume that $p$ is
$G$-elliptic, that is
\eqref{Gell} is satisfied.
Let $F[u]$ be of the form
\eqref{A5.2} (possibly with
some factors in the product
replaced by their conjugates)
and $f \in \sg.$ Assume
moreover that $u\in
H^s(\rd)$,
$s>\frac{d}{2}+\max\{|\rho_k|\}$,
is a solution of
\eqref{A5.1f}. In the case
$n=0$ assume further $\langle
x\rangle^{\epsilon_0} u\in
L^2(\rd)$ for some
$\epsilon_0>0$. Then
$u\in\sg$.
\end{theorem}
In fact we always assume that
$F[u]$ has the form in
\eqref{A5.2}, and we leave to
the reader the easy changes
when some factors of the
product in \eqref{A5.2} are
replaced by their conjugates.
\par We start by showing that, under the assumptions of Theorem \ref{AA5.1}, 
$u$ is in fact a Schwartz
function.
\begin{lemma}
\label{moreregularity}
Let $P$, $m$, $n$, $F[u]$, $f$ be as in Theorem \ref{AA5.1}. Let $u$ be a solution of \eqref{A5.1f} satisfying $\langle x\rangle^{\eps_0}u\in H^s(\rd)$ for some $s>\frac{d}{2}+\max\{|\rho_k|\}$, $\eps_0\geq0$. Then, when  $n>0$, we have  $\px^{\eps_0+\sigma_2}\langle D \rangle^{\sigma_1}u \in H^s(\R^d)$  for every $\sigma_1 \leq \min\{m,1\}$, and $\sigma_2 \leq \min\{n,1\}$. If $n=0$ and we assume in addition $\eps_0>0$, then we have
$\px^{\ve_0+\sigma_2}\langle D\rangle^{\sigma_{1}}
  u \in H^s(\R^d)$ for every $\sigma_1\leq \min\{m,1\}$ and
  $\sigma_2 \leq \min\{\ve_0,1\}$.
\end{lemma}

\begin{proof}
We first consider the case
$n>0$. Let
$E\in{\rm OP}G^{-m,-n}(\rd)$
be a parametrix for $P$;
hence $R:=EP-I\in{\rm
OP}G^{-1,-1}(\rd)$. We have
from \eqref{A5.1f}
\begin{equation}\label{iteration}
\px^{\eps_0+\sigma_2}\langle
D\rangle^{\sigma_1}u=
\px^{\eps_0+\sigma_2}\langle
D\rangle^{\sigma_1}Ef
-\px^{\eps_0+\sigma_2}\langle
D\rangle^{\sigma_1}Ru +
\px^{\eps_0+\sigma_2}\langle
D\rangle^{\sigma_1}EF[u].\end{equation}
Since $\sigma_1\leq1$ and
$\sigma_2\leq1$, the operator
$\px^{\sigma_2}\langle
D\rangle^{\sigma_1} R\in{\rm
OP}G^{0,0}(\rd)$ is bounded
on $H^{s,\eps_0}(\rd)$; cf. \eqref{contin}. Taking also into account the assumptions on $f$ it follows
therefore that the $H^s$-norm
of the first two terms in the
right-hand side of
\eqref{iteration} is finite.
Concerning the last term,
observe that, by the
assumptions on
$\sigma_1,\sigma_2$ and $h$, the operator
$\px^{\sigma_2}\langle
D\rangle^{\sigma_1}\circ E
\circ x^h \in {\rm
OP}G^{-m+\sigma_1,
-n+\sigma_2+|h|}(\R^d)$
belongs in fact to $ {\rm
OP}G^{-\max\{m-1,0\},0}(\R^d)$.
As a consequence, since
$M:=\max\{|\rho_k|\}\leq\max\{m-1,0\}$,
it is bounded
$H^{s-M,\eps_0}(\rd)\to H^{s,\eps_0}(\rd)$.
Hence by Schauder's estimates
(recall that $s>d/2+M$) we
have
\begin{align*}
\|\px^{\eps_0+\sigma_2}\langle D\rangle^{\sigma_1}EF[u]\|_s
&\leq C_s \sum_l\sum_{\rho_1,\ldots,\rho_l}\|
\px^{\eps_0}\prod_{k=1}^{\ell}\partial^{\rho_k}u\|_{s-M}\\
&\leq C'_s \sum_l\sum_{\rho_1,\ldots,\rho_l}\|
\px^{l\eps_0}\prod_{k=1}^{\ell}\partial^{\rho_k}u\|_{s-M}\\
&
\leq C''_s\sum_l\sum_{\rho_1,\ldots,\rho_l}
\prod_{k=1}^{\ell}\|\px^{\eps_0}\partial^{\rho_k}u\|_{s-M}
\leq C'''_s\sum_l
\|\px^{\eps_0}u\|_s^{\ell}<\infty.
\end{align*} We treat now the case
$n=0$, hence $h=0$ in the
nonlinearity \eqref{A5.2}.
We consider again \eqref{iteration}.
For the terms
$\px^{\ve_o+\sigma_2}\langle
D\rangle^{\sigma_1}Ef$ and
$\px^{\ve_o+\sigma_2}\langle
D\rangle^{\sigma_1}Ru $ we
argue as before. For the
nonlinear term observe that,
since $M:=\max\{|\rho_k|\}\leq\max\{m-1,0\}$,
we have that $\langle
D\rangle^{\sigma_1} E\in{\rm
OP}G^{-\max\{m-1,0\},0}(\rd)$
is bounded
$H^{s-M,l\eps_0}(\rd)\to
H^{s,l\eps_0}(\rd)$, for every $l$;
see \eqref{contin}. Hence for $\sigma_2\leq \ve_0\leq\ve_0(\ell-1)$ we get
\begin{eqnarray*}\|\px^{\ve_0+\sigma_2}
\langle
D\rangle^{\sigma_1}EF[u] \|_s&\leq& C_s\sum_l \sum_{\rho_1,\ldots,\rho_l}\|  \px^{l\eps_0}
\langle
D\rangle^{\sigma_1}E\prod_{k=1}^l\partial^{\rho_k}u \|_s\\
&\leq& C'_s \sum_l \sum_{\rho_1,\ldots,\rho_l}\|\px^{l\eps_0}
\prod_{k=1}^{\ell}
\partial^{\rho_k}u\|_{s-M}
\nonumber \\&\leq&
C''_s\sum_l \sum_{\rho_1,\ldots,\rho_l}\prod_{k=1}^{\ell}\|
\px^{\ve_0}
\partial^{\rho_k}u\|_{s-M}\nonumber\\&\leq&
 C'''_s\sum_l
\|\px^{\ve_0}u\|_s^{\ell}<\infty,
\end{eqnarray*} where we
applied again Schauder's
estimate and, in the last
inequality, \eqref{contin} to
$\partial^{\rho_k}$.

\end{proof}

Let us observe that, when $n>0$, an iterated application of Lemma
\ref{moreregularity} shows
that, under the assumptions
of Theorem \ref{AA5.1},
$\px^{\tau_2}\langle
D\rangle^{\tau_1}u \in
H^s(\R^d)$
 for every $\tau_1>0, \tau_2>0$, that is $u\in\cS(\rd)$. The same is true when $n=0$, because the
assumptions $u\in H^s(\rd)$,
$s>d/2+\max\{|\rho_k|\}$, and
$\px^{\ve_0}u\in L^2(\R^d)$, $\eps_0>0$, imply that for new
values of $s$ and $\eps_0$ as
above we have
$\px^{\ve_0}u\in H^s(\R^d)$ (see Remark \ref{rem1}), and Lemma \ref{moreregularity} still allows us to upgrade regularity and decay.\par
In particular, the sum $S_N^{s,\ve}[u]$ is finite for every $N \in \N.$
\par
In order to prove Theorem
\ref{AA5.1} it suffices to
verify that
$S^{s,\eps}_\infty[u]<\infty$,
in view of Proposition
\ref{stima0}. This will be
achieved by an iteration
argument involving the
partial sum of the series in
\eqref{accaenne2}, that is
\begin{equation}\label{accaenne}
S_N^{s,\eps}[f]=\sum_{|\al|+|\beta|\leq
N}\frac{\eab}{\mab}
\|x^\be\partial^\al f\|_s.
\end{equation}
We shall treat separately the cases $m\geq 1$ and $0<m<1,$ since the study of the nonlinearity
requires different arguments.

\subsection{Proof of Theorem \ref{AA5.1}: the case $m \geq 1$}
We need
several estimates to which we
address now.
\begin{proposition}\label{stima1}
Let $R\in{\rm
OP}G^{-1,-1}(\rd)$. Then for
every $s\in\R$ there exists a
constant $C_s>0$ such that,
for every $\eps\leq1$,
\[
\sum_{0<|\alpha|+|\beta|\leq N}\frac{\eab}{\mab}\|R(x^\beta\partial^\alpha
u)\|_s\leq C_s\eps
S^{s,\eps}_{N-1}[u].
\]
\begin{proof}
We first estimate the terms
with $\alpha=0$, hence
$\beta\not=0$. Let
$j\in\{1,\ldots,d\}$ such
that $\beta_{j}\not=0$. Since
$R\circ x_{j}\in{\rm
OP}G^{-1,0}(\rd)$ is bounded
on $H^s(\rd)$ we
have\footnote{We denote by
$e_j$ the $j$th vector of the
standard basis of $\rd$.}
\[
\frac{\eps^{|\beta|}}{|\beta|!}
\|R(x^\beta u)\|_s\leq
C_s\eps\frac{\eps^{|\beta|-1}}{|\beta|!}\|x^{\beta-e_{j}}u\|_s.
\]
Similarly one argues if
$\beta=0$, $\alpha\not=0$. If
finally $\alpha\not=0$,
$\beta\not=0$, hence for some $j,k\in\{1,\ldots,d\}$, we have
$\alpha_{j}\not=0$,
$\beta_{k}\not=0$, we
write
\[
x^\beta\partial^\alpha=\partial_{j}\circ
x_{k}
x^{\beta-e_{k}}
\partial^{\alpha-e_{j}}-\beta_{j}
x^{\beta-e_{j}}\partial^{\alpha-e_{j}}
\]
and use the fact that
$R\,\partial_{j}\circ
x_{k}\in{\rm
OP}G^{0,0}(\rd)$ is bounded
on $H^s(\rd)$. We get
\begin{multline*}
\frac{\eab}{\mab}\|R(x^\beta\partial^\alpha
u)\|_s\leq
C_s\eps^2\frac{\eps^{|\alpha|+|\beta|-2}}{\max\{|\alpha|,|\beta|\}!}\|x^{\beta-e_k}
\partial^{\alpha-e_{j}}
u\|_s\\
+C_s\eps^2\frac{
\eps^{|\alpha|+|\beta|-2}}{\max\{|\alpha|-1,|\beta|-1\}!}\|x^{\beta-e_{j}}
\partial^{\alpha-e_{j}}u\|_s,
\end{multline*}
(we understand that the second term in the right-hand side is omitted if $\beta_j=0$). These estimates give at once
the desired result if
$\epsilon\leq1$.
\end{proof}
\end{proposition}
\begin{proposition}\label{commutatore}
Let $P=p(x,D)$ be a
pseudodifferential with
symbol $p(x,\xi)$ satisfying
the estimates
\eqref{Gsymbols}, with
$m\geq0$, $n\geq0$. Let
$E\in{\rm OP}G^{-m,-n}(\rd)$.
Then for every $s\in\R$ there
exists a constant $C_s>0$
such that, for every
$\epsilon$ small enough and
$u\in \cS(\rd)$, we have
\begin{equation}\label{nn5}
\sum_{0<|\alpha|+|\beta|\leq N}\frac{\eab}{\mab}\|E[P,x^\beta
\partial^\alpha]u\|_s\leq C_s\eps
S^{s,\eps}_{N-1}[u].
\end{equation}
\end{proposition}
\begin{proof}
We
have
\[
[P,x^\beta
\partial^\alpha]=[P,x^\beta]\partial^\alpha+x^\beta[P,\partial^\alpha].
\]
Hence, using \eqref{pr1}, \eqref{pr2}, we get
\begin{multline}\label{c0}
[P,x^\beta
\partial^\alpha]u=\sum_{0\not=\gamma_0\leq\beta}
(-1)^{|\gamma_0|+1}\binom{\beta}{\gamma_0}(D^{\gamma_0}_\xi
p)(x,D)(x^{\beta-\ga_0}\partial^\alpha
u)\\
-
\sum_{0\not=\delta\leq\alpha}
\binom{\alpha}{\delta}
x^\beta\partial^\delta_x
p(x,D)
\partial^{\alpha-\delta} u.
\end{multline}
Given $\beta$, $\delta$, let
$\tilde{\delta}$ be a
multi-index of maximal length
among those satisfying
$|\tilde{\delta}|\leq
|\delta|$ and
$\tilde{\delta}\leq\beta$
(hence $|\tilde{\delta}|=|\delta|$ unless
$\beta-\tilde{\delta}=0$).
Writing $x^\beta= x^{\tilde{\delta}}x^{\beta-\tilde{\delta}}$
in the last term of
\eqref{c0} and using again
\eqref{pr1} we get
\begin{equation}\label{c0bis}
[P,x^\beta
\partial^\alpha]u=\sum_{\delta\leq\alpha}\sum_{\stackrel{\gamma_0\leq\beta-\tilde{\delta}}{
(\delta,\ga_0)\not=(0,0)}}(-1)^{|\gamma_0|+1}\binom{\beta-\tilde{\delta}}
{\gamma_0}
\binom{\alpha}{\delta}
x^{\tid}(D^{\ga_0}_\xi
\partial^\delta_x p)(x,D)(
x^{\beta-\tid-\ga_0}\partial^{\alpha-\delta}
u).
\end{equation}
We now look at the operator $
x^{\beta-\tid-\ga_0}\partial^{\alpha-\delta}$.
Given $\ga_0$, $\alpha$,
$\delta$, let $\tg_0$ be a
multi-index of maximal length
among those satisfying
$|\tg_0|\leq|\ga_0|$ and
$\tg_0\leq\alpha-\delta$. We
write, by the inverse Leibniz
formula \eqref{pr3},
\begin{multline}
x^{\beta-\tilde{\delta}-\ga_0}\partial^{\alpha-\delta}
=x^{\beta-\tilde{\delta}-\ga_0}\partial^{\tg_0}\partial^{\alpha-\delta-\tg_0}
=\partial^{\tg_0}\circ
x^{\beta-\tilde{\delta}-\ga_0}\partial^{\alpha-\delta-\tg_0}\\ +\sum_{0\not=\ga_1
\leq\beta-\tilde{\delta}-\ga_0\atop
\ga_1\leq\tilde{\ga}_0}\frac{(-1)^{|\ga_1|}(\beta-\tid-\ga_0)!}{(\beta-\tid-\ga_0-\ga_1)!}
\binom{\tg_0}{\ga_1}\partial^{\tg_0-\ga_1}\circ
x^{\beta-\tilde{\delta}-\ga_0-\ga_1}\partial^{\alpha-\delta-\tg_0}.
\end{multline}
We now look at the operator $
x^{\beta-\tid-\ga_0-\ga_1}\partial^{\alpha-\delta-\tg_0}$.
We denote by $\tg_1$ a
multi-index of maximal length
among those satisfying
$|\tg_1|\leq|\ga_1|$,
$\tg_1\leq
\alpha-\delta-\tg_0$. Again
by the inverse Leibniz
formula we have
\begin{multline}
x^{\beta-\tilde{\delta}-\ga_0-\ga_1}\partial^{\alpha-\delta-\tg_0}
=x^{\beta-\tilde{\delta}-\ga_0-\ga_1}\partial^{\tg_1}
\partial^{\alpha-\delta-\tg_0-\tg_1}\\
=\partial^{\tg_1}\circ
x^{\beta-\tilde{\delta}-\ga_0-\ga_1}\partial^{\alpha-\delta-\tg_0-\tg_1}
\\
+\sum_{0\not=\ga_2
\leq\beta-\tilde{\delta}-\ga_0-\ga_1\atop
\ga_2\leq\tg_1}\frac{(-1)^{|\ga_2|}(\beta-\tid-\ga_0-\ga_1)!}
{(\beta-\tid-\ga_0-\ga_1-\ga_2)!}
\binom{\tg_1}{\ga_2}\partial^{\tg_1-\ga_2}\circ
x^{\beta-\tilde{\delta}-\ga_0-\ga_1-\ga_2}\partial^{\alpha-\delta-\tg_0-\tg_1}.
\end{multline}
Continuing in this way and
substituting all in
\eqref{c0bis} we get
\begin{multline*}
[P,x^\beta
\partial^\alpha]u=\sum_{\delta\leq\alpha}\sum_{j=0}^h
\sum_{\gamma_0\leq\beta-\tilde{\delta}\
\atop
(\delta,\ga_0)\not=(0,0)}
\sum_{0\not=\ga_1\leq\beta-\tid-\ga_0\atop
\ga_1\leq\tg_0}\cdots
\sum_{0\not=\ga_j\leq\beta-\tid-\ga_0-\ldots-\ga_{j-1}\atop
\ga_j\leq \tg_{j-1}}
C_{\alpha,\beta,\delta,\ga_0,\ga_1,\ldots,\ga_j}\\
\times
p_{\alpha,\beta,\delta,\ga_0,\ga_1,\ldots,\ga_j}(x,D)\big(
x^{\beta-\tid-\ga_0-\ldots-\ga_j}\partial^{\alpha-\delta-\tg_0-\ldots-\tg_j}
u\big),
\end{multline*}
where $\tg_j$ is defined inductively as
a multi-index of maximal length among
those satisfying $|\tg_j|\leq |\ga_j|$
and $\tg_j\leq
\alpha-\delta-\tg_0-\ldots-\tg_{j-1}$,
\begin{align}\label{c1bis}
|C_{\alpha,\beta,\delta,\ga_0,\ga_1,\ldots,\ga_j}|&=\frac{\alpha!(\beta-\tid)!}
{(\alpha-\delta)!\delta!\ga_0!(\beta-\tid-\ga_0-\ldots-\ga_j)!}\prod_{k=1}^j\binom{\tg_{k-1}}{\ga_k}
\nonumber\\
&\leq
\frac{|\alpha|!|\beta-\tid|!}
{|\alpha-\delta|!\delta!\ga_0!|\beta-\tid-\ga_0-\ldots-\ga_j|!}
2^{|\tg_0+\ldots+\tg_{j-1}|},
\end{align}
cf. \eqref{pr5} and \eqref{pr6}, and
\begin{equation}
p_{\alpha,\beta,\delta,\ga_0,\ga_1,\ldots,\ga_j}(x,\xi)=x^{\tid}
\big(D^{\ga_0}_\xi\partial^\delta_x
p\big)(x,\xi)\xi^{\tg_0-\ga_1+\tg_1-\ldots-\ga_j+\tg_j},\quad
j\geq 0, \label{ordinecomm}
\end{equation}
(if $j=0$ in \eqref{c1bis} we
mean that there are not the
binomial factors, nor the
power of $2$).
 Observe that,
since we have $\ga_j\not=0$
for every $j\geq1$, this
procedure in fact stops after
a finite number of steps.\par
Now we observe that, by \eqref{Gsymbols}, 
\eqref{pr6}, 
and Leibniz' formula, for
every
$\theta,\sigma\in\mathbb{N}^d$
we have
\begin{equation}\label{c2}
|\partial^\theta_\xi\partial^\sigma_x
p_{\alpha,\beta,\delta,\ga_0,\ga_1,\ldots,\ga_j}(x,\xi)|\leq
C^{|\ga_0|+|\delta|+1}\ga_0!\delta!\langle
x\rangle^{n-|\sigma|}\langle\xi\rangle
^{m-|\theta|},
\end{equation}
for some constant $C$
depending only on $\theta$
and $\sigma$. In fact
$|\tid|\leq|\delta|$,
$|\tg_0-\ga_1+\tg_1-\ldots-\ga_j+\tg_j|\leq|\tg_0|\leq|\ga_0|$,
and the powers of $|\delta|$
and $|\gamma_0|$ which arise
can be estimated by
$C^{|\ga_0|+|\delta|+1}$ for
some $C>0$.\par
 We now use these last bounds to estimate
 $E\circ p_{\alpha,\beta,\delta,\ga_0,\ga_1,\ldots,\ga_j}(x,D)$. To this end,
 observe that this operator
 belongs to ${\rm OP}G^{0,0}(\rd)$, and
 therefore its norm as a
 bounded operator on
 $H^s(\rd)$ is estimated by a
 seminorm of its symbol in
 $G^{0,0}(\rd)$, depending only on
 $s$ and $d$. Such a seminorm
 is in turn estimated by the
 product of a seminorm of the
 symbol of $E$ in
 $G^{-m,-n}(\rd)$ and a
 seminorm of
 $p_{\alpha,\beta,\delta,\ga_0,\ga_1,\ldots,\ga_j}$
 in $G^{m,n}(\rd)$, again
 depending only on $s,d$.
 Hence we see from \eqref{c2}
 that
 \begin{equation}\label{c3}
 \|E\circ p_{\alpha,\beta,\delta,\ga_0,\ga_1,\ldots,\ga_j}(x,D)
 \|_{\mathcal{B}(H^s(\rd))}\leq
 C_s^{|\ga_0|+|\delta|+1}\ga_0!\delta!.
\end{equation}
Let now
$|\beta|\geq|\alpha|$. Then
$|\tid|=|\delta|$. Using
moreover the estimate
\[
\frac{|\alpha|!|\beta-\tilde{\delta}|!}{|\beta|!|\alpha-\delta|!}=\frac{|\alpha|!(|\beta|-|\delta|)!}{|\beta|!(|\alpha|-|\delta|)!}\leq1,
\]
together with \eqref{c1bis}
and \eqref{c3}, we
obtain
\begin{multline}\label{c4}
\frac{\eps^{|\alpha|+|\beta|}}{|\beta|!}
|C_{\alpha,\beta,\delta,\ga_0,\ga_1,\ldots,\ga_j}|
\|E\circ p_{\alpha,\beta,\delta,\ga_0,\ga_1,\ldots,\ga_j}(x,D)(
x^{\beta-\tid-\ga_0-\ldots-\ga_j}\partial^{\alpha-\delta-\tg_0-\ldots-\tg_j}
u)\|_s\\
\leq
C_s(C_s\epsilon)^{|\delta|+|\ga_0+\ldots+\ga_j|
+|\tg_0+\ldots+\tg_{j-1}|}
\frac{\eps^{|\alpha|+|\beta|-|\delta|
-|\ga_0+\ldots+\ga_j|-|\tg_0+\ldots+\tg_{j-1}|}}
{|\beta-\tid-\ga_0-\ldots-\ga_j|!}\\
\times
\|x^{\beta-\tid-\ga_0-\ldots-\ga_j}\partial^{\alpha-\delta-\tg_0-\ldots-\tg_j}
u\|_s.
\end{multline}
Similarly, if
$|\alpha|\geq|\beta|$ we have
$|\tg_k|=|\ga_k|$, $0\leq
k\leq j$, and
\begin{equation}\label{c5}
\frac{|\beta-\tilde{\delta}|!|\alpha-\delta-\tg_0-\ldots-\tg_j|!}{|\alpha-
\delta|!|\beta-\tid-\ga_0-\ldots-\ga_j|!
}\leq1,
\end{equation}
(recall that if
$|\tid|<|\delta|$ then
$\beta-\tid=\ga_0=\ldots=\ga_j=\tg_0=\ldots=\tg_j=0$).\par
By \eqref{c1bis}, \eqref{c3}
\eqref{c5}, we get in this
case
\begin{multline}\label{c4bis}
\frac{\eps^{|\alpha|+|\beta|}}{|\alpha|!}
|C_{\alpha,\beta,\delta,\ga_0,\ga_1,\ldots,\ga_j}|
\|E\circ p_{\alpha,\beta,\delta,\ga_0,\ga_1,\ldots,\ga_j}(x,D)(
x^{\beta-\tid-\ga_0-\ldots-\ga_j}\partial^{\alpha-\delta-\tg_0-\ldots-\tg_j}
u)\|_s\\
\leq
C_s(C_s\epsilon)^{|\delta|+|\ga_0+\ldots+\ga_j|+|\tg_0+\ldots+\tg_{j-1}|}
\frac{\eps^{|\alpha|+|\beta|-|\delta|
-|\ga_0+\ldots+\ga_j|-|\tg_0+\ldots+\tg_{j-1}|}}
{|\alpha-\delta-\tg_0-\ldots-\tg_j|!}\\
\times
\|x^{\beta-\tid-\ga_0-\ldots-\ga_j}\partial^{\alpha-\delta-\tg_0-\ldots-\tg_j}
u\|_s.
\end{multline}
Since, if
$|\beta|\geq|\alpha|$, we
have
\[
\max\{|\beta-\tid-\ga_0-\ldots-\ga_j|,
|\alpha-\delta-\tg_0-\ldots-\tg_j|\}=|\beta-\tid-\ga_0-\ldots-\ga_j|,
\]
whereas if
$|\alpha|\geq|\beta|$ it
turns out
\[
\max\{|\beta-\tid-\ga_0-\ldots-\ga_j|,
|\alpha-\delta-\tg_0-\ldots-\tg_j|\}=|\alpha-\delta-\tg_0-\ldots-\tg_j|,
\]
we deduce from \eqref{c4} and
\eqref{c4bis} that, if
$\epsilon$ is small enough,
\begin{align*}
&\sum_{|\alpha|+|\beta|\leq N}\frac{\eab}{\mab}\|E[P,x^\beta\partial^\alpha]u\|_s\\
&\leq C_s
\sum_{|\tilde{\alpha}|+|\tilde{\beta}|\leq
N-1}
\frac{\eps^{|\tilde{\alpha}|+|\tilde{\beta}|}}{\max\{|\tilde{\alpha}|,
|\tilde{\beta}|\}!}\|x^{\tilde{\beta}}\partial^{\tilde{\alpha}} u\|_s\sum_{j=0}^h\sum_{\delta}\sum_{\ga_1\not=0,
\ldots,\ga_j\not=0\atop
\ga_0:\,(\delta,\ga_0)\not=(0,0)}
 (C_s\epsilon)^{|\delta|+|\ga_0+\ga_1+\ldots+\ga_j|}\\
 &\leq S^{s,\eps}_{N-1}[u]\sum_{j=0}^h (C'_s\eps)^{j+1}
 \leq C''_s\eps
 S^{s,\eps}_{N-1}[u].
\end{align*}
\end{proof}

We now turn the attention to
the nonlinear term.
\begin{proposition}\label{nonlin1}
Let $E\in{\rm
OP}G^{-m,-n}(\rd)$, $m\geq1$,
$n\geq 0$, and
$h\in\mathbb{N}^d$, $|h|\leq
\max\{n-1,0\}$,
$\rho_k\in\mathbb{N}^d$,
$|\rho_k|\leq m-1$, for
$1\leq k\leq l$,
$l\geq2$.\par Then for every
$s>d/2+\max_k\{|\rho_k|\}$
there exists a constant
$C_s>0$ such that, for every
$\eps$ small enough and $u\in
\cS(\rd)$, the following
estimates hold:
\begin{multline}\label{c4tris}
\sum_{0<|\alpha|+|\beta|\leq
N}\frac{\eab}{\mab}\|E\big(
x^\beta\partial^\alpha\big(x^h \prod_{k=1}^l \partial^{\rho_k}u\big)\big)\|_s\\
\leq
C_s\epsilon\big(\|u\|_s^{l-1}S_{N-1}^{s,\eps}[u]+
(S_{N-1}^{s,\eps}[u])^{l}\big)
\end{multline}
if $n \geq 1$ and
\begin{multline}\label{nn8}
\sum_{0<|\alpha|+|\beta|\leq
N}\frac{\eab}{\mab}\|E\big(
x^\beta\partial^\alpha\prod_{k=1}^l \partial^{\rho_k}u\big)\|_s\\
\leq
C_{s}\epsilon\big(\|\langle
x\rangle^{\frac{1}{l-1}}
u\|_s^{l-1}
S_{N-1}^{s,\eps}[u]+
(S_{N-1}^{s,\eps}[u])^{l}\big).
\end{multline}
if $0\leq n<1.$
\end{proposition}
\begin{proof}
Let $n \geq 1,$ (hence
$|h|\leq n-1$). In the sum
\eqref{c4tris} we consider
the terms with $\alpha=0$.
Namely, we prove that
\begin{equation}\label{nn10}
 \sum_{0\not=|\beta|\leq N}
 \frac{\eps^{|\beta|}}{|\beta|!}\|E\big(
x^\beta x^h \prod_{k=1}^l
\partial^{\rho_k}u\big)\|_s\leq
C_s
\epsilon\|u\|_s^{l-1}S_{N-1}^{s,\eps}[u].
\end{equation}
Given $\beta\not=0$, let
$j\in\{1,\ldots,d\}$ such that
$\beta_j\not=0$. Since $E\circ x_j
x^h\in{\rm OP}G^{-m,-n+|h|+1}(\rd)$ is
bounded $H^{s-M}(\rd)\to H^s(\rd)$,
with $M=\max_k\{|\rho_k|\}$ (because
$|h|\leq n-1$, $|\rho_k|\leq m-1$), and
applying Schauder's estimates (recall
that $s-M>d/2$) we have
\[
\frac{\eps^{|\beta|}}{|\beta|!}\|E\big(
x^\beta x^h \prod_{k=1}^l
\partial^{\rho_k}u\big)\|_s\leq
C_s\eps\frac{\eps^{|\beta|-1}}{(|\beta|-1)!}\|x^{\beta-e_j}\partial^{\rho_1}
u\|_{s-M}\|u\|^{l-1}_s.
\]
Then \eqref{nn10} follows by
writing
\begin{equation}\label{nn11}
x^{\beta-e_j}\partial^{\rho_1}
u=\sum_{\gamma\leq\beta-e_j\atop\gamma\leq\rho_1}\frac{(-1)^{|\gamma|}(\beta-e_j)!}
{(\beta-e_j-\gamma)!}
\binom{\rho_1}{\gamma}\partial^{\rho_1-\gamma}(x^{\beta-e_j-\gamma}u).
\end{equation}
in view of\eqref{pr3}, and
using
\begin{equation}\label{nn12}
\frac{(\beta-e_j)!}{(|\beta|-1)!(\beta-e_j-\gamma)!}\leq\frac{1}
{(|\beta|-1-|\gamma|)!},
\end{equation}
cf. \eqref{pr5}.\par For $0\leq n <1$,
consider first the terms in \eqref{nn8}
with $\alpha=0$. We prove that
\[
 \sum_{0\not=|\beta|\leq N}
 \frac{\eps^{|\beta|}}{|\beta|!}\|E\big(
x^\beta \prod_{k=1}^l
\partial^{\rho_k}u\big)\|_s\leq
C_{s} \epsilon \|\langle
x\rangle^{\frac{1}{l-1}}
u\|_s^{l-1}
S_{N-1}^{s,\eps}[u].
\]
To this end, given
$\beta\not=0$, let
$j\in\{1,\ldots,d\}$ such
that $\beta_j\not=0$. Since
$E$ is bounded
$H^{s-M}(\rd)\to H^s(\rd)$, $M=\max_k\{|\rho_k|\}$,
by Schauder's estimates we
have
\begin{align*}
\frac{\eps^{|\beta|}}{|\beta|!}\|E\big(
x^\beta \prod_{k=1}^l
\partial^{\rho_k}u\big)\|_s&\leq
C'_s\eps
\frac{\eps^{|\beta|-1}}{(|\beta|-1)!}\|x^{\beta-
e_j}
\partial^{\rho_1}
u\|_{s-M}\|x_j
\prod_{k=2}^l\partial^{\rho_k} u\|_{s-M}\\
&\leq C''_{s}\eps
\frac{\eps^{|\beta|-1}}{(|\beta|-1)!}\|x^{\beta-
e_j}
\partial^{\rho_1} u\|_{s-M}\|\langle
x\rangle^{\frac{1}{l-1}}
u\|_s^{l-1},
\end{align*}
cf. the action on weighted Sobolev
spaces described in Section
\ref{prelimi}. Then the claim follows
by applying again \eqref{nn11} and
\eqref{nn12}.\par We now treat the
terms with $\alpha\not=0$ in the sums
\eqref{c4tris} and \eqref{nn8}. Namely,
we prove that (both in the cases $0\leq
n<1$ and $n\geq1$)
\begin{equation}\label{nn13}
\sum_{0<|\alpha|+|\beta|\leq
N\atop
\alpha\not=0}\frac{\eab}{\mab}\|E\big(
x^\beta\partial^\alpha\big(x^h
\prod_{k=1}^l
\partial^{\rho_k} u\big)\big)\|_s\leq
C_s\epsilon
(S_{N-1}^{s,\eps}[u])^{l}.
\end{equation}
Let
$\alpha\not=0$ and
$j\in\{1,\ldots,d\}$ such
that $\alpha_{j}\not=0$. We
can write
\[
x^\beta\partial^\alpha\big(x^h
\prod_{k=1}^l
\partial^{\rho_k} u\big)=Q_1^{\alpha,\beta}[u]+Q_2^{\alpha,\beta}[u],
\]
with
\[
Q_1^{\alpha,\beta}[u]=\partial_{x_{j}}
x^\beta\partial^{\alpha-e_{j}}\big(x^h
\prod_{k=1}^l
\partial^{\rho_k} u\big),
\]
\[
Q_2^{\alpha,\beta}[u]=-\beta_{j}
x^{\beta-e_{j}}
\partial^{\alpha-e_{j}}\big(x^h
\prod_{k=1}^l
\partial^{\rho_k} u\big).
\]
Now we estimate $\frac{\eab}{\mab}\|E
Q_1^{\alpha,\beta}[u]\|_s$. To this end
observe that, by Leibniz' formula,
\[
Q_1^{\alpha,\beta}[u]=\partial_{x_{j}}
\sum_{\delta_0+\delta_1+\ldots+\delta_{l}=\alpha-e_{j}\atop
\delta_0\leq h}
\frac{(\alpha-e_{j})!}{\delta_0!\delta_1!
\ldots\delta_{l}!}\frac{h!}
{(h-\delta_0)!}
x^{h-\delta_0}
x^\beta\prod_{k=1}^l
\partial^{\delta_k+\rho_k} u.
\]
 Let now $\tilde{\delta}_0$ be a
 multi-index of maximal length among
 those satisfying
 $|\tilde{\delta}_0|\leq|\delta_0|$
 and $\tilde{\delta}_0\leq\beta$.
 Observe that
 $E\partial_{x_{j_\alpha}}\circ
 x^{h-\delta_0} x^{\tilde{\delta}_0}\in{\rm OP}G^{-m+1,-n+|h|}(\rd)$ is bounded $
 H^{s-M}(\rd)\to H^{s}(\rd)$ with $M=\max_k\{|\rho_k|\}$
 (because
 $|\rho_k|\leq m-1$ for $1\leq k\leq l$ and $|h|\leq\max\{n-1,0\}\leq n$). Hence
\begin{multline*}
 \frac{\eab}{\mab}\|E
 Q_1^{\alpha,\beta}[u]\|_s\\
 \leq C_s \sum_{\delta_0+\delta_1+\ldots+\delta_{l}=\alpha-e_{j}\atop
\delta_0\leq h}
\frac{\eab}{\mab}\frac{(\alpha-e_{j})!}{\delta_0!\delta_1!
\ldots\delta_{l}!}\frac{h!}
{(h-\delta_0)!}\\
\times
 \|x^{\beta-\tilde{\delta}_0}\prod_{k=1}^l
\partial^{\delta_k+\rho_k} u\|_{s-M}.
\end{multline*}
We can now write
\begin{equation}\label{nn6}
x^{\beta-\tilde{\delta}_0}\prod_{k=1}^l
\partial^{\delta_k+\rho_k} u
=\prod_{k=1}^l x^{\ga_k}
\partial^{\delta_k+\rho_k} u,
\end{equation}
where
$\ga_1+\ldots+\ga_{l}=\beta-\tilde{\delta}_0$
and, if $|\beta|\leq
|\alpha|-1$, with
 $|\ga_k|\leq|\delta_k|$ for $1\leq
k\leq l$ (which is possible because in
that case
$|\beta-\tilde{\delta}_0|\leq|\alpha-\delta_0|-1$;
observe that if
$|\tilde{\delta}_0|<|\delta_0|$ then
$\beta-\tilde{\delta}_0=0$), whereas,
if $|\beta|\geq |\alpha|$, with
$|\ga_k|\geq|\delta_k|$ for $1\leq
k\leq l$ (which is possible because in
that case
$|\tilde{\delta}_0|=|\delta_0|$ and
$|\beta-\tilde{\delta}_0|\geq|\alpha-\delta_0|\geq|\alpha-\delta_0|-1$).\par
Hence we get by Schauder's estimates
\begin{multline} \label{nn1}
\frac{\eab}{\mab}\|E
Q_1^{\alpha,\beta}[u]\|_s\\
\leq C_s\eps
\sum_{\delta_0+\delta_1+\ldots+\delta_{l}
= \alpha-e_{j}\atop
\delta_0\leq h}
\prod_{k=1}^{l}
\frac{\eps^{|\gamma_k|+|\delta_k|}}
{\max\{|\gamma_k|,|\delta_k|\}!}\|x^{\ga_k}
\partial^{\delta_k+\rho_k}
u\|_{s-M},
\end{multline}
where if
$|\beta|\leq|\alpha|-1$ we
used the inequality
\begin{equation}\label{nn3}
\frac{1}{(|\alpha|-1)!}
\frac{(\alpha-e_{j})!}
{\delta_0!\delta_1!\ldots\delta_{l}!}
\leq\frac{1}{|\delta_0|!|\delta_1|!\ldots|\delta_{l}|!},
\end{equation}
which is \eqref{pr8}, whereas if $|\alpha|\leq|\beta|$ we applied
\begin{equation}\label{nn4}
\frac{1}{|\beta|!}
\frac{(\alpha-e_{j})!}
{\delta_0!\ldots\delta_{l}!}
\leq\frac{1}{|\tilde{\delta}_0|!|\ga_1|!
\ldots|\ga_{l}|!},
\end{equation}
which also follows at once from
\eqref{pr8}.\par Now, we write
$x^{\gamma_k}\partial^{\delta_k+\rho_k}u=\partial^{\rho_k}
\big(x^{\gamma_k}\partial^{\delta_k}
u\big)+[x^{\gamma_k}\partial^{\delta_k},\partial^{\rho_k}]u$
in the last term of \eqref{nn1}, so
that
\[
\|x^{\gamma_k}\partial^{\delta_k+\rho_k}u\|_{s-M}
\leq
\|x^{\gamma_k}\partial^{\delta_k}
u\|_s+\|\langle
D\rangle^{-|\rho_k|}[x^{\gamma_k}
\partial^{\delta_k},\partial^{\rho_k}]u\|_s.
\]
Using this last estimate we
get
\begin{multline*}
\frac{\eab}{\mab}\|E
Q_1^{\alpha,\beta}[u]\|_s\\
\leq C_s\eps
\sum_{\delta_0+\delta_1+\ldots+\delta_{l}=\alpha-e_{j}\atop
\delta_0\leq
h}\prod_{k=1}^{l}\frac{\eps^{|\gamma_k|+|\delta_k|}}
{\max\{|\gamma_k|,|\delta_k|\}!}\Big\{
\|x^{\ga_k}\partial^{\delta_k}
u\|_{s}+\sum_{|\gamma|\leq
m-1}\|\langle
D\rangle^{-|\gamma|}[x^{\gamma_k}
\partial^{\delta_k},\partial^\gamma]u\|_s\Big\},
\end{multline*}
(recall that the $\gamma_k$'s depend on
$\beta,\delta_1,\ldots,\delta_{l}$). We
now sum the above expression over
$|\alpha|+|\beta|\leq N$,
$\alpha\not=0$.  When $\alpha$ and
$\beta$ vary, every term in the above
product also appears in the development
of
\[
\Big\{\sum_{|\tilde{\alpha}|+|\tilde{\beta}|\leq
N-1}\frac{\eps^{|\tilde{\alpha}|+|\tilde{\beta}|}}
{\max\{|\tilde{\alpha}|,|\tilde{\beta}|\}!}\Big\{
\|x^{\tilde{\beta}}\partial^{\tilde{\alpha}}
u\|_{s}+\sum_{|\gamma|\leq
m-1}\|\langle
D\rangle^{-|\gamma|}[x^{\tilde{\beta}}
\partial^{\tilde{\alpha}},\partial^\gamma]u\|_s \Big\}\Big\}^{l},
\]
and is repeated at most, say, $L$ times, with $L$ depending only on $h$ and the dimension $d$. Hence we
obtain
\begin{align*}
&\sum_{|\alpha|+|\beta|\leq N\atop
\alpha\not=0}\frac{\eab}{\mab}\|E
Q_1^{\alpha,\beta}[u]\|_s\\
&\leq
C''_s\eps\Big\{\sum_{|\tilde{\alpha}|+|\tilde{\beta}|\leq
N-1}
\frac{\eps^{|\tilde{\alpha}|+|\tilde{\beta}|}}
{\max\{|\tilde{\alpha}|,|\tilde{\beta}|\}!}\Big\{
\|x^{\tilde{\beta}}\partial^{\tilde{\alpha}}
u\|_{s}+\sum_{|\gamma|\leq
m-1}\|\langle
D\rangle^{-|\gamma|}[x^{\tilde{\beta}}
\partial^{\tilde{\alpha}},\partial^\gamma]u\|_s \Big\}\Big\}^{l}\\
&\leq C''_s\eps\big\{
S^{s,\eps}_{N-1}[u]+C'''_s\eps
S^{s,\eps}_{N-2}[u]\big\}^{l}\leq
C''''_s\eps
(S^{s,\eps}_{N-1}[u])^{l},
\end{align*}
where we used Proposition \ref{commutatore} applied with
 $\partial^\gamma$ and $\langle D\rangle^{-|\gamma|}$ in place
 of $P$ and $E$ respectively, and we understand $S^{s,\eps}_{-1}[u]=0$.\par
We finally show that
\begin{equation}\label{nn7}
\sum_{|\alpha|+|\beta|\leq
N\atop
\alpha\not=0}\frac{\eab}{\mab}\|E
Q_2^{\alpha,\beta}[u]\|_s\leq
C_s\eps
(S^{s,\eps}_{N-1}[u])^{l}.
\end{equation}
Since the arguments are
similar to the previous ones,
we give only a sketch of the
proof.\par We can write
\[
Q_2^{\alpha,\beta}[u]=\beta_{j}
\sum_{\delta_0+\delta_1+\ldots+\delta_{l}=\alpha-e_{j}\atop
\delta_0\leq h}
\frac{(\alpha-e_{j})!}{\delta_0!\delta_1!
\ldots\delta_{l}!}\frac{h!}
{(h-\delta_0)!}
x^{h-\delta_0}
x^{\beta-e_j}\prod_{k=1}^l
\partial^{\delta_k+\rho_k} u.
\]
If $\beta_j\not=0$, we
choose a multi-index
$\tilde{\delta}_0$ of maximal
length among
 those satisfying
 $|\tilde{\delta}_0|\leq|\delta_0|$
 and $\tilde{\delta}_0\leq\beta-e_j$. Next, we use the fact that
 $E\circ
 x^{h-\delta_0} x^{\tilde{\delta}_0}\in{\rm OP}G^{-m,-n+|h|}(\rd)$ is
  bounded $
 H^{s-M}(\rd)\to H^{s}(\rd)$ with $M=\max_k\{|\rho_k|\}$ and we apply the
   decomposition
\[
x^{\beta-e_j-\tilde{\delta}_0}\prod_{k=1}^l
\partial^{\delta_k+\rho_k} u
=\prod_{k=1}^l x^{\ga_k}
\partial^{\delta_k+\rho_k}
u,
\]
with
$\ga_1+\ldots+\ga_l=\beta-e_j-\tilde{\delta}_0$,
and
   $|\gamma_k|\leq|\delta_k|$
   for
   $1\leq k\leq l$, if
   $|\beta|\leq|\alpha|$, or
   $|\gamma_k|\geq|\delta_k|$
   for
   $1\leq k\leq l$, if
   $|\alpha|\leq|\beta|$.
Moreover, if
$|\beta|\leq|\alpha|$ one
uses
\[
\frac{\beta_j}{|\alpha|}\frac{1}{(|\alpha|-1)!}
\frac{(\alpha-e_{j})!}
{\delta_0!\delta_1!\ldots\delta_{l}!}
\leq\frac{1}{|\delta_0|!|\delta_1|!\ldots|\delta_{l}|!},
\]
in place of \eqref{nn3},
whereas if
$|\alpha|\geq|\beta|$ (hence
$|\tilde{\delta}_0|=|\delta_0|$)
one uses
\[
\frac{\beta_j}{|\beta|}\frac{1}{(|\beta|-1)!}
\frac{(\alpha-e_{j})!}
{\delta_0!\ldots\delta_{l}!}
\leq\frac{1}{|\tilde{\delta}_0|!|\ga_1|!
\ldots|\ga_{l}|!},
\]
in place of \eqref{nn4}.
Therefore we get the same
formula \eqref{nn1} with
$Q_2^{\alpha,\beta}$ in place
of $Q_1^{\alpha,\beta}$. The
proof then proceeds as that
for $Q_1^{\alpha,\beta}$
without other modifications.
\end{proof}

We are now ready to conclude the proof of Theorem
\ref{AA5.1}. \\

\noindent
\textit{End of the proof of
Theorem \ref{AA5.1} (the case $m\geq 1$).} It
follows from \eqref{A5.1f}
that, for
$\alpha,\beta\in\N^d$,
$\epsilon>0$,
\[
\frac{\eab}{\mab}
x^\beta\partial^\alpha
Pu=\frac{\eab}{\mab}
x^\beta\partial^\alpha
 f +
 \frac{\eab}{\mab}
x^\beta\partial^\alpha F[u],
\]
so that
\begin{multline*}
\frac{\eab}{\mab}P(
x^\beta\partial^\alpha
u)=\frac{\eab}{\mab}[P,x^\beta\partial^\alpha]u+\frac{\eab}{\mab}
x^\beta\partial^\alpha
 f\\ +
 \frac{\eab}{\mab}
x^\beta\partial^\alpha F[u].
\end{multline*}
We now apply to both sides
the parametrix $E$ of $P$.
With $R=EP-I\in{\rm
OP}G^{-1,-1}(\rd)$ we get
\begin{multline*}
\frac{\eab}{\mab}x^\beta\partial^\alpha
u=-\frac{\eab}{\mab}
R(x^\beta\partial^\alpha
u)+\frac{\eab}{\mab}
E[P,x^\beta\partial^\alpha]u\\
+\frac{\eab}{\mab}E(x^\beta\partial^\alpha
f)+
\frac{\eab}{\mab}E(x^\beta\partial^\alpha F[u] ).
\end{multline*}
Taking the $H^s$ norms and
summing over
$|\alpha|+|\beta|\leq N$ give
\begin{eqnarray}\label{colonna}
S^{s,\epsilon}_N[u] &\leq&
\|u\|_s+\sum_{0<|\alpha|+|\beta|\leq
N}\frac{\eab}{\mab}\|R(x^\beta\partial^\alpha
u)\|_s  \\ \nonumber
&&+\sum_{0<|\alpha|+|\beta|\leq
N}\frac{\eab}{\mab}\|E[P,x^\beta\partial^\alpha]u\|_s
\\ &&+\sum_{0<|\alpha|+|\beta|\leq
N}\frac{\eab}{\mab}\|E(x^\beta
\partial^\alpha f)\|_s \nonumber \\
&&+\sum_{0<|\alpha|+|\beta|\leq
N}\frac{\eab}{\mab}\|E(x^\beta\partial^\alpha F[u] )\|_{s}. \nonumber
\end{eqnarray}
The second and the third term in the right-hand side of \eqref{colonna} can be estimated using Propositions \ref{stima1} and \ref{commutatore} while the term containing $f$ is obviously dominated by $S_{\infty}^{s,\ve}[f].$ For the last term we can apply Proposition \ref{nonlin1}. Hence, for $n \geq 1,$ we have that, for
$\eps$ small enough,
\begin{multline*}
S^{s,\epsilon}_N[u]\leq
\|u\|_s+ C_s
S^{s,\epsilon}_\infty[f]+C_s\epsilon\Big(
S^{s,\epsilon}_{N-1}[u]+
\sum_{l}\big((S^{s,\eps}_{N-1}[u])^{l}
+ \|u\|_s^{l-1}
S^{s,\eps}_{N-1}[u]\big)\Big),
\end{multline*}
whereas if $0\leq n<1$ we get
\begin{multline*}
S^{s,\epsilon}_N[u]\leq
\|u\|_s+C_s
S^{s,\epsilon}_\infty[f]+C_{s}\epsilon\Big(
S^{s,\epsilon}_{N-1}[u]+\sum_{l}\big(
(S^{s,\eps}_{N-1}[u])^{l}\\
+\|\langle
x\rangle^{\frac{1}{l-1}}u\|_s^{l-1}S^{s,\eps}_{N-1}[u]\big)\Big).
\end{multline*}
In both cases we obtain
$S^{s,\eps}_\infty[u]<\infty$
if $\epsilon$ is small
enough, which implies
$u\in\sg$ by Proposition
\ref{stima0} (or, more
simply, by the standard Sobolev
embeddings, since
$s>d/2$).\par \qed


\subsection{Proof of Theorem \ref{AA5.1}: the case $0 < m < 1$}
In this case the nonlinearity has the
form
\begin{equation}
\label{A5.2tris} F[u]=\sum_{h,l}
F_{h,l} x^{h}u^l,
\end{equation}
where $l\in\N$, $l\geq2$, $h \in
\N^{d},$ with $|h|\leq\max\{n-1,0\}$
and $F_{h,l}\in\mathbb{C}$, the above
sum being finite.\par We follow the
same argument used for the case $m\geq
1$. In particular we can estimate the
first four terms in the right-hand side
of \eqref{colonna} as before since
Propositions \ref{stima1} and
\ref{commutatore} hold in general for
$m >0.$ Hence, to conclude, it is
sufficient to prove an estimate for the
nonlinear term. We have the following
result.

\begin{proposition}
\label{Gammanonlinearestimates}
Let $P$ satisfy the
assumptions of Theorem
\ref{AA5.1} for $0<m<1$ and let
 $E$ be a parametrix of $P$. Let $l\in\mathbb{N}$,
  $l\geq2$, $h \in \N^{d}$, $|h|\leq \max\{n-1,0\}$. Then,
  there exists
a constant $C'_s>0$ and, for every
$\tau>0$, there exists
 $C_{\tau}>0$ such
 that, for every $\eps$ small
 enough and $u\in \cS(\rd)$
 we have
\begin{align} \label{nonlinestimate2}
\sum_{0<|\alpha+\beta|\leq N}
\frac{\ve^{|\alpha|+|\beta|}}{\max \{
|\alpha|, |\beta|\}!}& \| E
(x^{\beta}\partial^{\alpha}(x^{h}u^\ell))
\|_{s} \leq  \tau C'_{s} \|
u\|_{s}^{\ell-1} S_{N}^{s,\ve}[u]
\nonumber
\\&+ C'_{s} (\ve
C_{\tau}+\tau+\ve)
(S_{N-1}^{s,\ve}[u])^{\ell}+C'_{s}\ve
\|\langle x
\rangle^{\frac{1}{\ell-1}}
u\|^{\ell-1}_{s}
S_{N-1}^{s,\ve}[u].
\end{align}
\end{proposition}

\begin{proof} We
first consider the terms in
\eqref{nonlinestimate2} with
$\alpha=0$. Since $E\circ x^{h}\in{\rm
OP}G^{-m,0}(\rd)$ is bounded on
$H^s(\rd)$, we
 have, by Schauder's estimates:
\[
\| E(x^{h+\beta}u^{\ell})\|_{s}\leq
C'_{s}\|x^{\beta}u^{\ell}\|_{s}\leq
C''_{s}\|x^{\beta-e_{j}}u\|_{s}\cdot
\|x_{j}u^{\ell-1}\|_{s}
\] if, say,
$\beta_j\not=0$. Then we get
\begin{equation}\label{s0}\sum_{0<|\beta|\leq N}
\frac{\ve^{|\beta|}}{ |\beta|!}
\|E(x^{\beta}u^\ell)\|_{s} \leq
C'''_{s} \eps\|
\px^{\frac{1}{\ell-1}}u\|_{s}^{\ell-1}
\cdot S_{N-1}^{s,\ve}[u].
\end{equation} Consider now the terms
in \eqref{nonlinestimate2} with
 $\alpha\not=0$. We may write
\begin{eqnarray} x^{\beta}\partial^{\alpha}(x^{h}u^{\ell})&=&x^{h+\beta}\partial^{\alpha}(u^{\ell}) +
\sum_{\stackrel{0\neq \gamma\leq \alpha}{\gamma\leq h}}\binom{h}{\gamma}\frac{\alpha!}{(\alpha-\gamma)!}x^{h+\beta-\gamma}\partial^{\alpha-\gamma}(u^{\ell}) \nonumber \\ &=&x^{h}\partial_{j}(x^{\beta}\partial^{\alpha-e_{j}}(u^{\ell}))-\beta_{j}x^{h+\beta-e_{j}}\partial^{\alpha-e_{j}}(u^{\ell}) \nonumber \\ &&+
\sum_{\stackrel{0\neq \gamma\leq \alpha}{\gamma\leq h}}\binom{h}{\gamma}\frac{\alpha!}{(\alpha-\gamma)!}x^{h+\beta-\gamma}\partial^{\alpha-\gamma}(u^{\ell}).
\end{eqnarray}
Since $E$ is bounded $H^{s-m}(\rd)\to
H^s(\rd)$, we then obtain
\begin{eqnarray}\label{mario}
 \| E (x^{\beta}\partial^{\alpha}(x^{h}u^\ell)) \|_{s} &\leq&
C'_{s}\|\partial_{j}(x^{\beta}\partial^{\alpha-e_{j}}(u^{\ell}))\|_{s-m}+C'_{s}\beta_{j}
\|x^{\beta-e_{j}}\partial^{\alpha-e_{j}}(u^{\ell})\|_{s-m}
\nonumber \\ &&+C'_{s}
\sum_{\stackrel{0\neq \gamma\leq
\alpha}{\gamma\leq
h}}\binom{h}{\gamma}\frac{\alpha!}{(\alpha-\gamma)!}\|x^{\beta-\gamma}\partial^{\alpha-\gamma}(u^{\ell})\|_{s-m}.
\end{eqnarray}
Let us estimate the first term in the right-hand side of \eqref{mario}.
We observe that for every $\tau >0$
there exists a constant $C_{\tau}>0$ such that
$$\langle \xi \rangle^{-m}
|\xi_{j}|\leq \tau |\xi_{j}|
+ C_{\tau}.$$ Hence
\begin{eqnarray*}
\| \partial_{j}(x^{\beta}\partial^{\alpha-e_{j}}(u^{\ell}))
\|_{s-m}&=&  \|\langle D \rangle^{-m}\partial_{j}(x^{\beta}\partial^{\alpha-e_{j}}(u^{\ell}))\|_{s}
\\ &\leq& \tau  \|
\partial_{j}(x^{\beta}\partial^{\alpha-e_{j}}(u^{\ell}))\|_{s}
+ C_\tau\|
x^{\beta}\partial^{\alpha-e_{j}}(u^{\ell})\|_{s}
\\ &\leq&\tau
\beta_{j}\|x^{\beta-e_{j}}\partial^{\alpha-e_{j}}(u^{\ell})\|_{s}
+ \tau \|x^{\beta}\partial^{\alpha}(u^{\ell})\|_{s}
\\ &&+ C_\tau
\|x^{\beta}\partial^{\alpha-e_{j}}
(u^{\ell})\|_{s}.
\end{eqnarray*}
Now we replace
$$\partial^{\alpha}(u^{\ell})= \ell u^{\ell-1}\partial^{\alpha}u  +
\sum_{\delta_1+\ldots+\delta_{\ell}=\alpha\atop\delta_{k}\neq
\alpha\, \forall
k}\frac{\alpha!}{\delta_{1}! \ldots
\delta_{\ell}!}
\prod_{k=1}^l\partial^{\delta_{k}}u$$
in the last estimate and we come back
to \eqref{mario}. We get, for a new
constant $C'_s>0$,
\begin{eqnarray} \label{eco}
\| E(x^{\beta}\partial^{\alpha}(x^{h}u^{\ell}))
\|_{s}&
 \leq& C'_{s}C_\tau\|x^{\beta}\partial^{\alpha-e_{j}}
 (u^{\ell})\|_{s}
  + C'_{s}(1+\tau)\beta_{j}\|x^{\beta-e_{j}}
  \partial^{\alpha-e_{j}}(u^{\ell})
  \|_{s} \nonumber \\ &&+ \tau C'_{s} \|u \|_{s}^{\ell-1}
   \|x^{\beta}
  \partial^{\alpha}u\|_{s} \nonumber \\ &&+ \tau C'_{s}
  \sum_{\delta_1+\ldots+\delta_{\ell}=\alpha\atop\delta_{k}\neq
  \alpha\,
  \forall k}\frac{\alpha!}{\delta_{1}! \ldots \delta_{\ell}!}
   \|x^{\beta}\prod_{k=1}^l\partial^{\delta_{k}}u\|_{s} \nonumber \\ &&+C'_{s} \sum_{\stackrel{0\neq \gamma\leq \alpha}{\gamma\leq h}}\binom{h}{\gamma}\frac{\alpha!}{(\alpha-\gamma)!}\|x^{\beta-\gamma}\partial^{\alpha-\gamma}(u^{\ell})\|_{s-m}.
   \end{eqnarray}
We have now to estimate the terms in the right-hand
side of \eqref{eco}.
Concerning the first one,
applying Leibniz' formula we
obtain
\[\|x^{\beta}\partial^{\alpha-e_j}(u^{\ell})\|_s
\leq\sum_{\delta_1+\ldots+
\delta_\ell=\alpha-e_j}\frac{(\alpha-e_j)!}{\delta_1!\ldots
\delta_{\ell}!}
\|x^{\beta}\prod_{k=1}^l
\partial^{\delta_k}u\|_s.
\] If $|\beta|\geq |\alpha|,$
then we can argue as in the
previous section and find
$\gamma_1, \ldots,
\gamma_{\ell} \in \N^d$ such
that $\gamma_1 + \ldots
+\gamma_\ell=\beta$ and
$|\gamma_k|\geq |\delta_k|$
for every $k=1,\ldots, \ell.$
Moreover, we observe that the
following estimate holds:
\begin{equation}
\label{primastima}\frac{1}{|\beta|!} \cdot \frac{(\alpha-e_j)!}
{\delta_1!\ldots \delta_\ell !}
\leq \frac{1}{|\gamma_1|!\ldots |\gamma_\ell|!}.
\end{equation}
Then
\begin{multline}\label{s1}
\sum_{|\alpha|+|\beta|\leq
N\atop 0<|\alpha|\leq
|\beta|}
\frac{\ve^{|\alpha|+|\beta|}}{\max\{|\alpha|,
|\beta|\}!} \|
x^{\beta}\partial^{\alpha-e_j}(u^{\ell})\|_s
\leq C''_s \ve
\sum_{|\alpha|+|\beta|\leq
N\atop 0<|\alpha|\leq
|\beta|}
\sum_{\delta_1+\ldots
+\delta_\ell=\alpha-e_j}\prod_{k=1}^\ell
\frac{\ve^{|\gamma_k|+|\delta_{k}|}}{|\gamma_k|!}\\
\times\|x^{\gamma_k}\partial^{\delta_k}u\|_s
\leq C'''_s \ve
(S_{N-1}^{s,\ve}[u])^\ell.
\end{multline}
On the other hand, for
$|\beta|\leq |\alpha|-1$ we
can choose multi-indices
$\gamma_1,\ldots,
\gamma_\ell$ such that
$\gamma_1+\ldots+\gamma_\ell=\beta$
and $|\gamma_k|\leq
|\delta_k|$ for any $k
=1,\ldots, \ell$ and observe
that
\begin{equation}\label{secondastima}
\frac{1}{(|\alpha|-1)!}
\frac{(\alpha-e_j)!}{\delta_1! \ldots
\delta_\ell !} \leq
\frac{1}{|\delta_1|!\ldots
|\delta_\ell|! }.
\end{equation}
Then
\begin{multline}\label{s2}
\sum_{|\alpha|+|\beta|\leq
N\atop 0<|\beta|\leq
|\alpha|-1}\frac{\ve^{|\alpha|+|\beta|}}{\max\{|\alpha|,
|\beta|\}!} \|
x^{\beta}\partial^{\alpha-e_j}(u^{\ell})\|_s
\\
\leq C''_s
\sum_{|\alpha|+|\beta|\leq
N\atop 0<|\beta|\leq
|\alpha|-1}
\sum_{\delta_1+\ldots+
\delta_\ell=\alpha-e_j}
\prod_{k=1}^{\ell}
\frac{\ve^{|\gamma_k|+|\delta_k|}}{|\delta_k|!}\|x^{\gamma_k}
\partial^{\delta_k}u\|_s
 \leq C'''_s \ve
(S_{N-1}^{s,\ve}[u])^{\ell},
\end{multline}
for new constants $C''_s$ and
$C'''_s$.\par For the second term in
the right-hand side of \eqref{eco} we
can argue as before, with
$\ga_1+\ldots+\ga_\ell=\beta-e_j$ and
$|\ga_k|\leq|\delta_k|$ for
$k=1,\ldots,\ell$, if
$|\beta|\leq|\alpha|$ or
$|\ga_k|\geq|\delta_k|$ for
$k=1,\ldots,\ell$, if
$|\beta|\geq|\alpha|$, using the
estimates
$$\frac{\beta_j}{|\beta|} \frac{1}{(|\beta|-1)!}\frac{(\alpha-e_j)!}{\delta_1!\ldots \delta_\ell!}\leq \frac{1}{|\gamma_1|!\ldots |\gamma_\ell|!}$$
respectively,
$$\frac{\beta_j}{|\alpha|}\frac{1}{(|\alpha|-1)!}\frac{(\alpha-e_j)!}{\delta_1!\ldots \delta_\ell!}\leq \frac{1}{|\delta_1|!\ldots |\delta_\ell|!}$$
instead of
\eqref{primastima},
respectively
\eqref{secondastima}. We
obtain, for a new constant
$C'_s>0$,
\begin{equation} \label{s3}\sum_{|\alpha|+|\beta|\leq N\atop\alpha \neq 0}
\frac{\ve^{|\alpha|+|\beta|}}{\max\{|\alpha|, |\beta|\}!} \| \beta_j
  x^{\beta-e_j}\partial^{\alpha-e_j}(u^{\ell})\|_s \leq C'_s \ve
  (S_{N-1}^{s,\ve}[u])^{\ell}.\end{equation}
Concerning the fourth term in
\eqref{eco}, we can decompose similarly
$\beta=\ga_1+\ldots+\ga_\ell$ and argue
as before, taking into account that now
$|\ga_1+\ldots+\ga_\ell+\delta_{1}+\ldots+\delta_{\ell}|=|\alpha+\beta|$,
so that we do not longer gain $\eps$ as
a factor in the estimate. Hence we get
\begin{equation} \label{s4}
\sum_{|\alpha+\beta|\leq N\atop\alpha
\neq
0}\frac{\ve^{|\alpha|+|\beta|}}{\max\{|\alpha|,
|\beta|\}!}\sum_{\delta_1+\ldots+
\delta_\ell=\alpha\atop
\delta_k\not=\alpha\,\forall
k}\frac{\alpha!}{\delta_1!\ldots
\delta_{\ell}!}
\|x^{\beta}\prod_{k=1}^\ell\partial^{\delta_k}u\|_s
\leq
C'_{s}(S_{N-1}^{s,\ve}[u])^{\ell}\end{equation}
for a new constant $C'_s$. Finally, for
the last term in \eqref{eco}, we first
observe that $\max
\{|\alpha-\gamma|,|\beta-\gamma|\}
=\max \{|\alpha|, |\beta|\}-|\gamma|.$
Then we can argue as before obtaining
the estimate
\begin{equation}\label{s5}
\sum_{|\alpha+\beta|\leq
N\atop\alpha \neq
0}\frac{\ve^{|\alpha|+|\beta|}}{\max\{|\alpha|,
|\beta|\}!} \sum_{\stackrel{0\neq \gamma\leq \alpha}{\gamma\leq h}}\binom{h}{\gamma}\frac{\alpha!}{(\alpha-\gamma)!}\|x^{\beta-\gamma}\partial^{\alpha-\gamma}(u^{\ell})\|_{s-m}\leq C_{s}\ve(S_{N-1}^{s,\ve}[u])^{\ell}.
\end{equation}
The estimates \eqref{s0}, \eqref{s1},
\eqref{s2}, \eqref{s3},
\eqref{s4}, \eqref{s5} applied in
\eqref{eco} yield
\eqref{nonlinestimate2}.
\end{proof}

\textit{End of the proof of Theorem
\ref{AA5.1} (the case $0<m<1$).} Using the same argument as in the case $m\geq 1$, by Propositions \ref{stima0}, \ref{stima1}, \ref{commutatore},
\ref{Gammanonlinearestimates} we obtain
\begin{multline*}
S_{N}^{s,\ve}[u] \leq \|u
\|_{s}+ C'_{s}
S_{\infty}^{s,\ve}[f]+
C'_s\ve
S_{N-1}^{s,\ve}[u]+\sum_l\Big(\tau
C'_{s} \| u\|_{s}^{\ell-1}
 S_{N}^{s,\ve}[u]
\nonumber
\\+ C'_{s} (\ve
C_{\tau}+\tau+\ve)
(S_{N-1}^{s,\ve}[u])^{\ell}+C'_{s}\ve
\|\langle x
\rangle^{\frac{1}{\ell-1}}
u\|^{\ell-1}_{s}
S_{N-1}^{s,\ve}[u]\Big)
\end{multline*}
for every $N\geq1$
 and $\eps$ small enough. Now, choosing
$\tau < (2\sum_l C'_{s}
\|u\|_{s}^{\ell-1})^{-1}$ we obtain
\begin{multline*}S_{N}^{s,\ve}[u] \leq 2\|u
\|_{s}+ 2C'_{s}
S_{\infty}^{s,\ve}[f]+
2C'_s\ve S_{N-1}^{s,\ve}[u]+
\\
+\sum_l\Big( 2C'_{s} (\ve C_{\tau}+\tau+\ve)
(S_{N-1}^{s,\ve}[u])^{\ell}+2C'_{s}\ve
\|\langle x
\rangle^{\frac{1}{\ell-1}}
u\|^{\ell-1}_{s}
S_{N-1}^{s,\ve}[u]\Big).
\end{multline*}
 Then we can
iterate the last estimate
observing that, shrinking
$\tau$ and then $\ve$, the
quantity $\ve C_{\tau}+\tau+\ve$
can be taken arbitrarily
small. This gives
$S^{s,\ve}_\infty[u]<\infty$
and therefore $u\in\sg$.

\section{Remarks and applications}\label{secesempi}
\subsection{Lower a priori regularity}\label{lowerapriori}
For special nonlinearities, in Theorem \ref{mainthm} we
can assume lower a priori regularity on the solution $u$.
 For example, if $F[u]=(\partial^{\rho}u)^l$, $|\rho|
 \leq\min\{m-1,0\}$, $l\in\mathbb{N}$, $l\geq2$, or even
   $F[u]=|\partial^{\rho}u|^{l-1}\partial^\rho u$,
    $|\rho|\leq\min\{m-1,0\}$, $l\in\mathbb{N}$, $l>2$
     odd (as in \eqref{KG}), then we can assume
     $u\in H^s(\rd)$, with
     $s>\frac{d}{2}-\frac{m-|\rho|}{l-1}$, $s\geq|\rho|$.
      Indeed, such a solution is actually in
      $H^\infty(\rd)$; see e.g. \cite[Lemma 4.1,
      Remark 4.1]{A3} (where that threshold is also
       proved to be sharp). We also refer to
       \cite{A3} for  other types of non-linearities.
\subsection{Eigenfunctions of $G$-elliptic operators}
In the linear case, the assumptions on the a priori regularity of $u$ can be
relaxed assuming $u \in \mathcal{S}'(\R^d).$ Then, we have the following result.
\begin{theorem}\label{linearthm}
Let $P$ be a G-elliptic
pseudodifferential operator with a
symbol $p(x,\xi)$ satisfying
\eqref{Gsymbols}. Then there exists
$\eps>0$ such that every solution $u
\in \mathcal{S}'(\R^d)$ of the equation
$Pu=0$ extends to a holomorphic
function in the sector of
$\mathbb{C}^d$
$$\mathcal{C}_\eps=\{z=x+iy\in\mathbb{C}^d:\
|y|\leq\eps(1+|x|)\},
$$ satisfying there the
estimates \eqref{in0bis} for
some constants $C>0, c>0$.
\end{theorem}
\begin{proof}
It follows from the existence
of a parametrix (see Section
\ref{prelimi}) that any solution
$u\in\cS'(\rd)$ of $Pu=0$ is in fact a
Schwartz function. Hence we can
apply Theorem \ref{mainthm} directly
without using Lemma \ref{moreregularity}.
Moreover it follows from the
classical Fredholm theory of
{\it globally} regular operators
(see e.g. \cite[Theorem
3.1.6]{nicola}) that the
kernel of $P$ is a finite
dimensional subspace of
$\cS(\rd)$, which implies
that there exists a
 sector where all the
 solutions extend
 holomorphically.
\end{proof}

The main application of
Theorem \ref{linearthm}
concerns eigenfunctions of
$G$-elliptic operators of
orders $m>0$, $n>0$. Indeed,
in that case, if $P$ is
$G$-elliptic also $P-\lambda$
is $G$-elliptic for every $\lambda\in\C$,
and one can apply Theorem
\ref{linearthm} to
$P-\lambda$. As regards
existence, we recall that if
$P\in {\rm OP}G^{m,n}(\rd)$,
$m>0$, $n>0$, is formally
self-adjoint (i.e. symmetric
when regarded as an operator in $L^2(\rd)$
with domain $\cS(\rd)$) then
it has a sequence of real
eigenvalues either diverging
to $+\infty$ or $-\infty$,
and $L^2(\rd)$ has an
orthonormal basis made of
eigenfunctions of $P$ (cf.
e.g \cite{MP} or
\cite[Theorem
4.2.9]{nicola}). As an
example in dimension 1, one
can consider the operator
$Pu=-(1+x^2)u+x^2 u-2xu'$,
$x\in\R$.
\subsection{Solitary waves}
The present subsection is
devoted to some applications
to solitary waves, in
particular to the proof of
Theorem \ref{applthm}. First
we report the following
useful characterization of the condition
\eqref{simbest}.
\begin{proposition}\label{carat}
The estimates
\eqref{simbest}{} are
equivalent to requiring that
$p(\xi)$ extends to a
holomorphic function
$p(\xi+i\eta)$ in a sector of
the type \eqref{sector}, and
satisfies there the bound
$|p(\xi+i\eta)|\leq
C'\langle\xi\rangle^{m}$.
\end{proposition}
\begin{proof}
The sufficiency of \eqref{simbest} for
the holomorphic extension with the
desired bound follows exactly as in the
last part of the proof of Theorem
\ref{estensione}, where in \eqref{cla6}
the exponential factor is now replaced
by $\langle \xi\rangle^{m}$.\par In the
opposite direction, we obtain
\eqref{simbest} from Cauchy's estimates
applied to a disc in $\mathbb{C}$ with
center at $\xi$ and radius
$\eps'\langle \xi\rangle$ for some
small $\eps'>0$ (independent of $\xi$).
\end{proof}

\begin{proof}[Proof of
theorem \ref{applthm}]
Consider first the equation \eqref{kdvt}. We observe that
$u$ satisfies the equation
\begin{equation}\label{eqsol}
Mu+(V-1)u=F[u],
\end{equation}
cf. the proof of \cite[Theorem
3.2.1]{BL2}. By \eqref{plm0} and the
condition $V>1$ the symbol of the
linear part of the equation 
\eqref{eqsol}, that is $p(\xi)+V-1$, is
$G$-elliptic: for some constant $c>0$
\begin{equation}\label{plm}
p(\xi)+V-1\geq c\langle
\xi\rangle^m,\quad \xi\in\R.
\end{equation}
Moreover by \eqref{simbest} it
satisfies the analytic symbol estimates
\eqref{Gsymbols} (with $n=0$). To
conclude the proof it is sufficient to
show that $u\in H^s(\rd)$ and $\langle
x\rangle^{\eps_0} u\in L^2(\rd)$ for
some $\eps_0>0$, $s>d/2$, and to apply
Theorem \ref{mainthm}. The fact that
$u$ enjoys the above properties will
follow from \cite[Theorem 3.1.2,
Corollary 4.1.6]{BL2} once we observe
that the function
\[
K(\xi)=\frac{1}{p(\xi)+V-1}
\]
satisfies $|K(\xi)|\leq C\langle\xi\rangle^{-m}$ for some $C>0$ and belongs to $H^\infty(\rd)$.
This is clear, because \eqref{plm} and \eqref{simbest} give
\[
|\partial^\alpha K(\xi)|\leq C_\alpha\langle \xi\rangle^{-m-|\alpha|},
\]
and $m\geq1$.\par The case of the
equation \eqref{lwt} is
completely similar: in place
of \eqref{plm} one just has
$V\, Mu+V-1=F[u],$ and the
above arguments apply to the
function
$K(\xi)=(Vp(\xi)+V-1)^{-1}$.\par
The theorem is then proved.
\end{proof}

As an example where the
solutions are known in closed
form, consider the
generalized Korteweg-de Vries
equation
\begin{equation}\label{kdvgen}
v_t+v_x+v^{l} v_x+v_{xxx}=0,
\end{equation}
where $l\geq1$ is a positive
integer. Here we have
$p(\xi)=\xi^2$. The solitary
wave solutions have the form
$v(x,t)=u(x-Vt)$, where $V>1$
and
\[
u(x)=\sqrt[l]{\frac{(l+1)(l+2)(V-1)}{2}}{\rm
Cosh}^{-2/{l}}\Big(\frac{\sqrt{V-1}}{2}lx\Big),
\]
which has poles at the points
$z=i\frac{(2k+1)\pi}{l\sqrt{V-1}}$,
$k\in\mathbb{Z}$. Also, the
exponential decay in sectors
containing the real axis predicted
by Theorem \ref{applthm} is
confirmed.\par During the
years 1990-2000, several
papers were devoted to $5$-th
order and $7$-th order
generalization of KdV, see
for example Porubov
\cite[Chapter 1]{A32}. The
corresponding stationary
equation is of the type
\begin{equation}
\label{A4.4} \sum_{j=0}^m a_j
u^{(j)} +Q[u]=0,
\end{equation}
where $Q$ is a polynomial,
$Q[u] = \sum_{j=2}^M b_ju^j$
and $a_0\neq 0.$ Because
of physical assumptions, the
equation $\sum_{j=0}^m a_j
\lambda^j=0$ has no purely
imaginary roots, and then all
the solutions of the
corresponding linear equation
have exponential
decay/growth. This condition
can be read as
$G$-ellipticity of the symbol
of the linear part of the
corresponding stationary
equation: $\sum_{j=0}^m
a_j(i\xi)^j\not=0$ for
$\xi\in\R$, in particular
$\xi^2+V-1\not=0$ in the case
of\eqref{kdvgen}. Non-trivial solutions
$u$ of \eqref{A4.4} with
$u(x) \rightarrow 0$ as $x
\rightarrow \pm \infty$ may
exist or not, according to
the coefficients $a_j, b_j,$
and when they exist, in
general they do not have an
explicit analytic expression.
Holomorphic extension and
exponential decay on a sector
are granted anyhow by Theorem \ref{applthm}.

\subsection{Standing wave solutions of the Schr\"odinger equation} Consider the
Schr\"odinger equation in
$\rd$,
\[
i\partial_t v+\Delta v=\mu
|v|^{l-1}v,\quad (t,x)\in\R\times\rd,
\]
with $l\in\mathbb{N}$, $l>2$ odd,
$\mu\in\mathbb{C}$, and look at
standing wave solutions, i.e.
$v(t,x)=e^{i\omega t} u(x)$,
$\omega>0$. The corresponding equation
for $u$ is
\[
\Delta u-\omega u=\mu
|u|^{l-1}u.
\]
Since the operator $\Delta-\omega$ is
$G$-elliptic (because $\omega>0$), when
solutions $u$ exist, with $u\in
H^s(\rd)$,
$s>\frac{d}{2}-\frac{2}{l-1}$,
$s\geq0$, and $\langle
x\rangle^{\eps_0} u\in L^2(\rd)$, for
some $\eps_0>0$, then Theorem
\ref{mainthm} and the remark in Subsection
\ref{lowerapriori} assure that $u$
extends to a holomorphic function on a
sector of the type \eqref{in6} and
displays there an exponential decay of type \eqref{in0bis}.
This applies, in particular, to the
bound states in $H^1(\rd)$ exhibited in
\cite{A35} when $l<\frac{d+2}{d-2}$,
$d\geq3$.
\subsection{Sharpness of the
results} Here we show the
sharpness of Theorem
\ref{mainthm} as far as the
shape of the domain of
holomorphic extension is concerned.\par
Consider in dimension $d=1$
the equation
\[
-u''+e^{-2i\theta}
u=\frac{e^{-2i\theta}}{2}u^2,
\]
where $-\pi<\theta\leq\pi$,
$|\theta|\not=\frac{\pi}{2}$.
This equation is
$G$-elliptic, since it is
elliptic and
$\xi^2+e^{-2i\theta}\not=0$
for every $\xi\in\R$. An
explicit Schwartz solution is
given by
\[
u(x)=3{\rm
Cosh}^{-2}\Big(\frac{e^{-i\theta}}{2}x\Big).
\]
The function $u$ extends to a
meromorphic function in the
complex plane with poles at
$z=e^{i(\theta+\pi/2)}(2k+1)\pi$,
$k\in\mathbb{Z}$. This shows
that in Theorem \ref{mainthm}
we cannot replace the sector
\eqref{in6}, e.g., with a
larger set of the type
\[
\{z=x+iy\in\mathbb{C}^d:\
|y|\leq\eps(1+|x|)\psi(x)\},
\]
for any continuous function
$\psi(x)>0$, with $\psi(x)\to+\infty$
as $|x|\to+\infty$.

\section*{Acknowledgments}
We wish to thank Piero D'Ancona, Todor
Gramchev, Luigi Rodino, Enrico Serra
and Paolo Tilli for helpful discussions
about several issues related to the
subject of this paper.


\begin{thebibliography}{10}
\bibitem{albert} P. Albert,
J. L. Bona and J.-C. Saut,
{\it Model equations for
waves in stratified fluids},
Proc. R. Soc. Lond. A {\bf
453} (1997), 1233--1260.
\bibitem{amick} C. J. Amick
and J. F. Toland, {\it
Homoclinic orbit in the
dynamic phase-space analogy
of an elastic strut}, Euro.
J. Appl. Math., {\bf 3}
(1992), 97--114.
\bibitem{benjamin} T. B.
Benjamin, J. L. Bona and D.
K. Bose, {\it Solitary wave
solutions of nonlinear
problems}, Phil. Trans. R.
Soc. Lond. A {\bf 331}
(1990), 195--244.
\bibitem{A35} H. Berestycki and P.-L. Lions, {\it Nonlinear scalar field equations I, II}, Arch. Rational Mech. Anal., {\bf 82} (1983), 313--375.
\bibitem{A3} H. A. Biagioni and T. Gramchev, {\it Fractional derivative estimates in Gevrey classes, global regularity and decay for solutions to semilinear equations in $\R^n$}, J. Differential Equations, {\bf 194} (2003), 140--165.
\bibitem{bo1} J. Bona and Z. Grujic',
{\it Spatial analyticity properties of
nonlinear waves}. Dedicated to Jim
Douglas, Jr. on the occasion of his
75th birthday. Math. Models Methods
Appl. Sci.,  {\bf 13} (2003), 345--360.
\bibitem{bo2} J. Bona, Z. Grujic' and H. Kalisch {\it
Algebraic lower bounds for the uniform
radius of spatial analyticity for the
generalized KdV equation},
  Ann. Inst.
H. Poincar\'e Anal. Non Lin\'eaire, {\bf
22} (2005), 783--797.
\bibitem{BL1}
J. Bona and Y. Li, {\it Analyticity of solitary-wave solutions of model equations for long waves}.  SIAM J. Math. Anal.  \textbf{27}  (1996),  n. 3, 725--737.
\bibitem{BL2} J. Bona and Y. Li, {\it Decay and analyticity of solitary waves}, J. Math. Pures Appl., {\bf 76} (1997), 377--430.
\bibitem{bo3} J. Bona and F.B. Weissler,
{\it Pole dynamics of interacting
solitons and blowup of complex-valued
solutions of KdV},  Nonlinearity, {\bf
22} (2009), 311--349.
\bibitem{A18} M. Cappiello, T. Gramchev and L. Rodino, {\it Super-exponential decay and holomorphic extensions for semilinear equations with polynomial coefficients}, J. Funct. Anal., {\bf 237} (2006), 634--654.
\bibitem{A38} M. Cappiello, T. Gramchev and L. Rodino, {\it Exponential decay and regularity for $SG$-elliptic operators with polynomial coefficients}, in ``Hyperbolic problems and regularity questions'', Trends Math., Birkh\"auser, Basel, 2007, 49--58.
\bibitem{A39} M. Cappiello, T. Gramchev and L. Rodino, {\it Semilinear pseudo-differential equations and travelling waves}, Fields Institute Communications, {\bf 52} (2007), 213--238.
\bibitem{CGR} M. Cappiello, T. Gramchev and L. Rodino, {\it Sub-exponential decay and uniform holomorphic extensions for semilinear pseudodifferential equations}. Comm. Partial Differential Equations \textbf{35} (2010). To appear.
\bibitem{CGR2} M. Cappiello, T. Gramchev and L. Rodino, {\it Entire extensions and 
 exponential decay for semilinear elliptic equations}, J. Anal. Math. \textbf{111} (2010). To appear.
\bibitem{rD2}  H. O. Cordes, {\it The technique of pseudodifferential operators}, London Math. Soc. Lecture Notes Ser. {\bf 202}, Cambridge University Press, Cambridge, 1995.
\bibitem{ES}
Y. V.\ Egorov and B.-W. Schulze, \textit{Pseudo-differential operators, singularities, applications}  Operator Theory: Advances and Applications, \textbf{93}, Birkh\"auser Verlag, Basel, 1997.
\bibitem{F-G} H. G. Feichtinger and K. H. Gr\"ochenig,  {\it
Banach spaces related to integrable group representations
 and their atomic decompositions. II}.  Monatsh. Math.,  {\bf 108}  (1989), 129--148.
\bibitem{GS}
I. M. Gel'fand and  G. E. Shilov, {\it Generalized functions II}, Academic Press, New York, 1968.
\bibitem{gramchev} T.
Gramchev, {\it Perturbative methods in
scales of Banach spaces: applications
for Gevrey regularity of solutions to
semilinear partial differential
equations}. Microlocal analysis and
related topics.  Rend. Sem. Mat. Univ.
Politec. Torino, {\bf 61} (2003),
1--134.
\bibitem{gru} Z. Grujic' and H.
Kalisch, {\it Local well-posedness of
the generalized Korteweg-de Vries
equation in spaces of analytic
functions}, Differential Integral
Equations,  {\bf 15} (2002), 1325--1334.
\bibitem{hay} N. Hayashi, {\it
Solutions of the (generalized)
Korteweg-de Vries equation in the
Bergman and the Szeg\"o spaces on a
sector}, Duke Math. J., {\bf 62} (1991),
575--591.
\bibitem{hormanderIII} L.\ H{\"o}rmander, {\it The analysis of linear partial differential
operators, III}, Springer-Verlag, 1985.
\bibitem{kato} T. Kato and K. Masuda, {\it
Nonlinear evolution equations and
analyticity. I},
  Ann. Inst. H.
Poincar\'e Anal. Non Lin\'eaire {\bf 3}
(1986), 455--467.
\bibitem{MP} L. Maniccia and P. Panarese, \textit{Eigenvalue asymptotics for a class of md-elliptic $\psi$do's on manifolds with cylindrical exits}, Annali Mat. Pura Appl., \textbf{181} (2002), 283--308.
\bibitem{Me}
R. Melrose, \textit{Geometric scattering theory}, Stanford Lectures. Cambridge Univ. Press, Cambridge, 1995.
\bibitem{nicola} F. Nicola
and L. Rodino, {\it Global
Pseudodifferential Calculus on
Euclidean Spaces}, Birkh\"auser, Basel,
2010, to appear.
\bibitem{rD1} C. Parenti,
{\it Operatori
pseudo-differentiali in
$\mathbb{R}^n$
 e applicazioni},
  Annali Mat. Pura Appl., {\bf 93} (1972), 359--389.
  \bibitem{A32} A. V. Porubov, {\it Amplification of nonlinear strain in solids}, World Scientific, Singapore, 2003.

\bibitem{Schrohe:2}
E. Schrohe, \textit{Spaces of weighted symbols and weighted Sobolev spaces on manifolds}. In ``Pseudodifferential Operators", Proceedings Oberwolfach 1986. H. O. Cordes, B. Gramsch and H. Widom editors, Springer LNM, \textbf{1256} New York, 360--377  (1987).
\bibitem{Schrohe:1}
{E. Schrohe}, \emph{{Complex
powers on noncompact
manifolds and manifolds with
  singularities}}, {Math. Ann.}, \textbf{281} (1988), no.~3, 393--409.

\bibitem{triebel} H. Triebel,
{\it Interpolation theory, function spaces,
 differential operators}. Second edition. Johann
 Ambrosius Barth, Heidelberg, 1995.

\bibitem{weinstein} M.
Weinstein, {\it Existence and
dynamic stability of solitary
wave solutions of equations
arising in long wave
propagation}, Comm. Partial
Differential Equations, {\bf
12} (1987), 1133--1173.
\end{thebibliography}
\end{document}